\documentclass[preprint]{elsarticle}
\usepackage[utf8]{inputenc}
\usepackage{amsmath, amssymb, amsthm, amsfonts}
\usepackage{bm}
\usepackage{titlesec}
\usepackage{mathtools}
\usepackage{stmaryrd}
\usepackage{tikz, graphicx}
\usepackage{caption, subcaption}
\usepackage{hyperref, cleveref}
\usepackage[titletoc]{appendix}
\usepackage[normalem]{ulem}
\usepackage{geometry}

\usepackage{multirow}

\usetikzlibrary{positioning}

\theoremstyle{definition}

\theoremstyle{lemma}
\newtheorem{lemma}{Lemma}
\theoremstyle{theorem}
\newtheorem{theorem}{Theorem}
\theoremstyle{assumption}
\newtheorem{assumption}{Assumption}

\newcommand{\td}[2]{\frac{{\rm d}#1}{{\rm d}{ {#2}}}} 
\newcommand{\pd}[2]{\frac{\partial#1}{\partial#2}}

\newcommand{\nor}[1]{\left\| #1 \right\|} 
\newcommand{\LRp}[1]{\left( #1 \right)} 
\newcommand{\LRb}[1]{\left| #1 \right|} 
\newcommand{\LRc}[1]{\left\{ #1 \right\}} 
\newcommand{\LRl}[1]{\left. #1 \right|} 
\newcommand{\jump}[1] {\ensuremath{\left\llbracket#1\right\rrbracket}} 
\newcommand{\avg}[1] {\ensuremath{\LRc{\!\!\LRc{#1}\!\!}}}
\newcommand{\prodmean}[1] {\ensuremath{\LRp{\!\LRp{#1}\!}}}
\newcommand{\fnt}[1]{\bm{\mathsf{ #1}}}

\newcommand{\note}[1]{{\color{blue}{#1}}}
\newcommand{\bnote}[1]{{#1}}
\newcommand{\rnote}[1]{{#1}}

\newtheorem{remark}{Remark}

\begin{document}

\title{High order entropy stable schemes for the quasi-one-dimensional shallow water and compressible Euler equations} 
\author[1]{Jesse Chan}
\author[2]{Khemraj Shukla}
\author[3]{Xinhui Wu}
\author[1]{Ruofeng Liu}
\author[1]{Prani Nalluri}

\affiliation[1]{Department of Computational Applied Mathematics and Operations Research, Rice University, Houston, TX, 77005}
\affiliation[2]{Division of Applied Mathematics, Brown University, Providence, RI 02906}
\affiliation[3]{Matroid, Inc., Palo Alto, CA, 94306}

\begin{abstract}
High order schemes are known to be unstable in the presence of shock discontinuities or under-resolved solution features for nonlinear conservation laws. Entropy stable schemes address this instability by ensuring that physically relevant solutions satisfy a semi-discrete entropy inequality independently of discretization parameters. This work extends high order entropy stable schemes to the quasi-1D shallow water equations and the quasi-1D compressible Euler equations, which model one-dimensional flows through channels or nozzles with varying width. 

We introduce new non-symmetric entropy conservative finite volume fluxes for both sets of quasi-1D equations, as well as a generalization of the entropy conservation condition to non-symmetric fluxes. When combined with an entropy stable interface flux, the resulting schemes are high order accurate, conservative, and semi-discretely entropy stable. For the quasi-1D shallow water equations, the resulting schemes are also well-balanced.
\end{abstract}

\maketitle

\section{Introduction}

Computational fluid dynamics simulations increasingly require higher resolutions for a variety of applications \cite{pi2013cfd}. For certain flows, high order accurate numerical methods are more accurate per degree of freedom compared to low order methods, and provide one avenue towards high accuracy while retaining reasonable efficiency \cite{wang2013high}. 
In this work, we extend high order entropy stable numerical schemes to the quasi-1D shallow water and the quasi-1D compressible Euler equations. An example of a quasi-1D domain is illustrated in Fig.\ \ref{fig:quasi_1d_domains}. Such systems are often used to model one-dimensional flows in domains with varying width, such as channels \cite{bermudez1994upwind} or nozzles \cite{corberan1995tvd, kroner2005numerical}. These systems have the simplicity of 1D equations, but incorporate effects from spatially varying domain widths. 

\begin{figure}[!h]
\centering
\includegraphics[width=.6\textwidth]{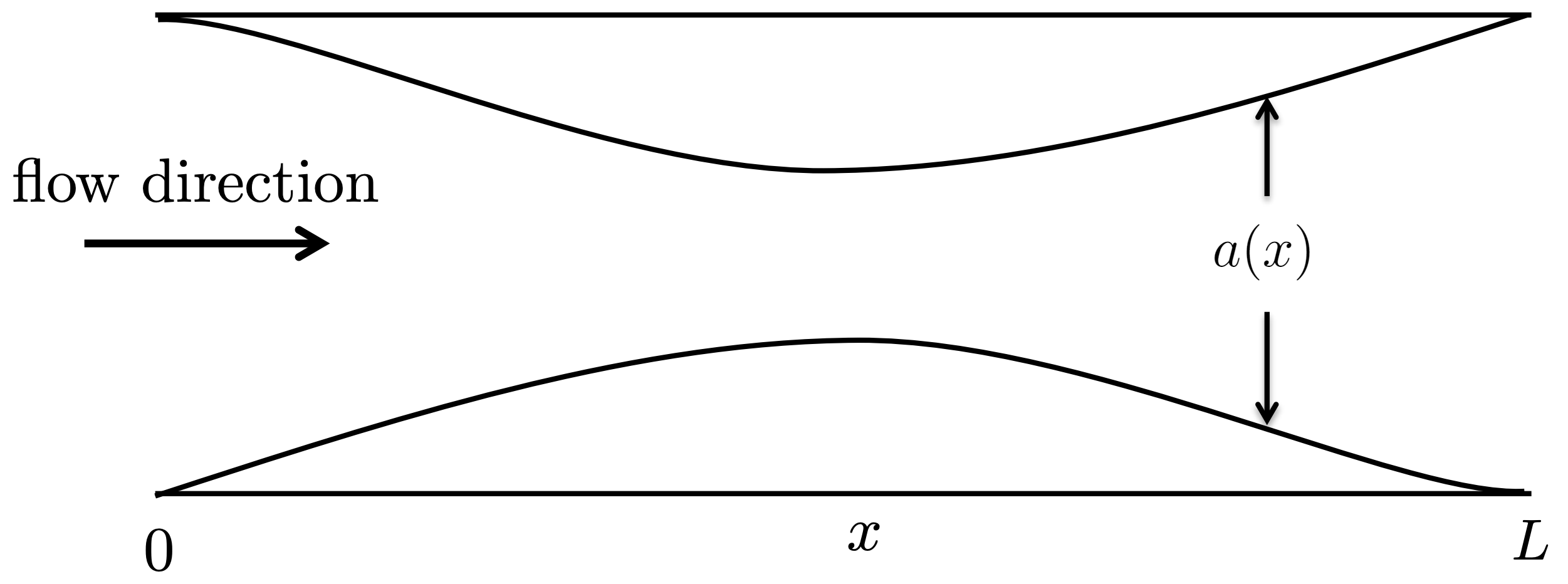}
\caption{An example of a quasi-1D domain with spatially varying channel width $a(x)$ from \cite{carlberg2018conservative}.}
\label{fig:quasi_1d_domains}
\end{figure}

We first review relevant literature for each system of quasi-1D equations. For the quasi-1D shallow water equations, early work focused on the well-balanced property \cite{bermudez1994upwind, vazquez1999improved, garcia2000numerical} for the nonlinear shallow water equations with general channel widths \cite{balbas2009central, hernandez2011shallow, murillo2014accurate}. More recent work has introduced high order accurate discretizations that preserve both well-balancedness and positivity \cite{xing2016high, qian2018positivity} for open channels with variable widths. For the quasi-1D compressible Euler equations, theoretical work includes studies on weak solutions and entropic properties \cite{le1989shock, kroner2008minimum}, while numerical schemes for the quasi-1D compressible Euler equations have included well-balanced schemes \cite{kroner2005numerical, kroner2008minimum, clain2009first} and a variety of treatments of non-conservative source terms which arise during discretization \cite{helluy2012well, gascon2021numerical}.

In this work, we consider schemes for the quasi-1D shallow water and compressible Euler equations which are entropy stable in addition to being high order accurate and well-balanced. These schemes are based on a ``flux differencing'' formulation \cite{fjordholm2012arbitrarily, fisher2013high, gassner2016split, chen2017entropy, chan2018discretely} and satisfy a semi-discrete dissipation of entropy independently of the approximation error. The main contributions of this paper are the construction and analysis of new entropy conservative fluxes which lie at the heart of ``flux differencing'' formulations. In particular, the derived fluxes include non-symmetric terms, which correspond to non-conservative terms in the quasi-1D system of equations. For brevity, we focus on 1D discretizations based on diagonal-norm summation-by-parts (SBP) operators \cite{gassner2013skew, svard2014review}. However, the \rnote{main theoretical tool that we introduce (a non-symmetric generalization of Tadmor's entropy conservative flux condition \cite{tadmor1987numerical}) is straightforward to extend} to more general entropy stable discretizations \cite{parsani2016entropy, chan2018discretely} or multi-dimensional discretizations \cite{wu2021entropy}. 

\section{Entropy analysis for quasi-1D systems}

In this work, we focus on two common quasi-1D systems: the quasi-1D shallow water equations and the quasi-1D compressible Euler equations. These equations are commonly used to model channel flow \cite{hernandez2011shallow} and flow through a nozzle \cite{kroner2008minimum}, respectively. For both quasi-1D systems, it is possible to show that the entropy variables are the same as for the original 1D system, and that the entropy inequality is simply the original entropy inequality scaled by the spatially varying channel width.

\subsection{The quasi-1D shallow water equations} 

First, we consider the 1D shallow water equations in symmetric open channels with varying channel width and bathymetry. This type of domain leads to additional terms involving the gradient of the width and the depth of the channel. The quasi-1D shallow water equations \cite{qian2018positivity, hernandez2011shallow} can be written as follows:
\begin{align}
\frac{\partial }{\partial t}\begin{bmatrix}
ah\\
ahu
\end{bmatrix}+
\frac{\partial }{\partial x}\begin{bmatrix}
ahu\\
ahu^2\\
\end{bmatrix}
+ 
\begin{bmatrix}
0\\
gah\pd{}{x}\LRp{h+b}
\end{bmatrix} = 0.
\label{eq:SWE_quasi_1D}
\end{align}
Here, $h$ is the water height, $b$ is the bottom bathymetry, $u$ is the water velocity, and $a(x)$ denotes the width of the channel at a point $x$. We assume the channel width does not change over time, such that $\frac{\partial }{\partial t} a = 0$. 

These equations reduce to the standard 1D shallow water equations when $a(x)$ is spatially constant. However, the spatial dependence of the width $a(x)$ requires the modification and redesign of 1D entropy stable numerical fluxes for quasi-1D equations to achieve entropy conservation on the discrete level. Furthermore, for the shallow water equations, our numerical schemes should remain well-balanced, such that they preserve the steady state solution $h + b = \text{constant}, u = 0$ for a non-flat bottom topography $b$.

\subsection{The quasi-1D compressible Euler equations}

We will also consider the quasi-one-dimensional compressible Euler equations with symmetric varying nozzle width as follows \cite{courant1999supersonic, giles2001analytic, kroner2008minimum, helluy2012well}:
\begin{align}
\frac{\partial }{\partial t}\begin{bmatrix}
a\rho\\
a\rho u\\
aE
\end{bmatrix}+
\frac{\partial }{\partial x}\begin{bmatrix}
a\rho u\\
a\rho u^2\\
au(E+p)
\end{bmatrix} + 
a\frac{\partial }{\partial x}\begin{bmatrix}
0\\
p\\
0
\end{bmatrix}= 0.
\label{eq:Euler_q1D}
\end{align}
Here, $\rho$ and $p$ denote density and the pressure, respectively. The velocity in the $x$ direction is denoted by $u$. The total energy is denoted by $E$ and satisfies the constitutive relation involving the pressure $p$
\begin{align}
E = \frac{1}{2}\rho u^2 + \frac{p}{\gamma-1},
\end{align}
where $\gamma = 1.4$ is the ratio of specific heat \note{of} a diatomic gas. Again, we assume that the width, $a(x)$, does not change over time.

\subsection{Continuous entropy analysis for quasi-1D systems}

In this section, we introduce entropy-flux pairs for the quasi-1D shallow water and compressible Euler equations. One can derive from the two-dimensional equations that the entropy, entropy flux, and entropy potential for each quasi-1D system are the relevant quantities for the standard 1D system scaled by $a(x)$. For example, let $\bm{u} = [ah,\: ahu]^T$ denote the scaled conservative variables for the shallow water equations. The entropy $S$, entropy flux $F$, and entropy potential, $\psi$ are then given as
\begin{align*}
S(\bm{u}) = \frac{1}{2}ahu^2+\frac{1}{2}gah^2 + gahb, \qquad F(\bm{u}) = \frac{1}{2}a h u^2 + agh^2u + gahbu, \qquad \psi = \frac{1}{2} gah(h+b)u.
\end{align*}
Similarly, for the compressible Euler equations, we define the scaled conservative variables $\bm{u} = [a\rho,\: a\rho u,\: aE]^T$. The entropy $S$, entropy flux $F$, and entropy potential $\psi$ are:
\begin{align*}
S(\bm{u}) = \frac{- a \rho s}{\gamma -1}, \qquad s = \log\LRp{\frac{p}{\rho^\gamma}}, \qquad F(\bm{u}) = \frac{-a\rho u s}{\gamma -1}, \qquad \psi = a\rho u,
\end{align*}
\rnote{Note that $S(\bm{u})$ are entropies for the original 1D shallow water and 1D compressible Euler equations scaled by $a$. We first note that we can prove $S(\bm{u})$ is a convex entropy with respect to $\bm{u}$. Additionally, if the conservative variables and entropy are both scaled by $a(x)$ for the quasi-1D shallow water and compressible Euler equations, the entropy variables for the quasi-1D shallow water and compressible Euler equations are the same as the entropy variables for the standard shallow water and compressible Euler equations. This is a consequence of the following lemma:
\begin{lemma}
Let $\eta(\bm{q})$ denote a differentiable convex entropy with respect to $\bm{q}$. Let $\bm{u} = a\bm{q}$, where $a(x) > 0$ is some scalar function, and let $S(\bm{u}) = a\eta(\bm{u} / a)$. Then, $S(\bm{u})$ is a differentiable convex function with respect to $\bm{u}$ and $\pd{S}{\bm{u}} = \pd{\eta}{\bm{q}}$.
\end{lemma}
\begin{proof}
The convexity of the quasi-1D definitions of $S(\bm{u})$ is a consequence of the fact that $\eta(\bm{q})$ is convex with respect to $\bm{q}$, which implies that $\eta(a\bm{q})$ is convex with respect to $\bm{u} =  a\bm{q}$ since convexity is preserved under affine mappings. The remainder of the proof follows by the chain rule:
\[
\pd{S(\bm{u})}{\bm{u}} = a\pd{\eta\LRp{\bm{u}/a}}{\bm{u}} =  a\pd{\eta}{\bm{q}} \cdot \pd{\bm{u}/a}{\bm{u}} = a\pd{\eta}{\bm{q}}\frac{1}{a} = \pd{\eta}{\bm{q}}.
\]
\end{proof}
This lemma implies that, if the quasi-1D entropy is simply the $a$-scaled entropy of the original 1D system, then the entropy variables for the quasi-1D version of a system of nonlinear conservation laws are identical to the entropy variables of the original 1D system.
}

%

For these entropy flux pairs, solutions of the quasi-1D shallow water and compressible Euler equations satisfy an entropy inequality (for appropriate boundary conditions) \cite{castro2013entropy, kroner2008minimum}
\[
\pd{S}{t} + \pd{F}{x} \leq 0.
\]
We provide derivations of entropy conservation for sufficiently regular solutions in \ref{appendix:entropy_analysis}.

\section{Entropy conservative numerical fluxes for quasi-1D systems}

The main contribution of this work is the construction of numerical schemes which mimic the entropy stability of the continuous system. Entropy conservative numerical fluxes in the sense of Tadmor \cite{tadmor1987numerical} are a key component of such schemes, and we derive new entropy conservative fluxes for the quasi-1D shallow water and compressible Euler equations. However, because quasi-1D equations are not conservative systems, the resulting fluxes are no longer symmetric. Motivated by this fact, we introduce an alternative definition of an entropy conservative flux:
\begin{assumption}
Throughout this work, we will assume that the flux $\bm{f}_{EC}$ satisfies an entropy conservation condition:
\begin{align}
\fnt{v}_L^T\bm{f}_{EC}(\fnt{u}_L,\fnt{u}_R) - \fnt{v}_R^T\bm{f}_{EC}(\fnt{u}_R,\fnt{u}_L) &= \psi(\fnt{u}_L) - \psi(\fnt{u}_R). \label{eq:nonsym_cond_LR}
\end{align}
We will refer to an entropy stable flux as any flux $\bm{f}^*$ which satisfies an entropy dissipation condition: 
\begin{align}
\fnt{v}_L^T\bm{f}^*(\fnt{u}_L,\fnt{u}_R) - \fnt{v}_R^T\bm{f}^*(\fnt{u}_R,\fnt{u}_L) &\geq \psi(\fnt{u}_L) - \psi(\fnt{u}_R). \label{eq:nonsym_cond_LR_ES}
\end{align}
We do not assume that these fluxes satisfy consistency or symmetry conditions. 
\end{assumption}
Note that \bnote{(\ref{eq:nonsym_cond_LR})} reduces to the standard Tadmor condition if the flux is symmetric \cite{tadmor1987numerical}. It is also possible to treat the quasi-1D equations as nonlinear hyperbolic systems in non-conservative form using the framework of \cite{castro2013entropy, renac2019entropy, waruszewski2022entropy}. However, due to the simple structure of the non-conservative terms in the quasi-1D shallow water and compressible Euler equations, we opt instead for a more direct approach to proving consistency and entropy stability in this paper.

\subsection{The quasi-1D shallow water equations}

We propose the following numerical fluxes with bathymetry based on the entropy conservative fluxes from \cite{chen2017entropy}:
\begin{align}
f_h &= \avg{ahu} \nonumber\\
f_{hu} &= \avg{ahu}\avg{u} + \boxed{\frac{g}{2}a_L h_L(h_R+b_R)}. \label{eq:SWE_quasi_1D_flux}
\end{align}
These fluxes provide consistent and symmetric approximations of all conservative terms in the shallow water equations, but also introduce non-symmetric terms which involve the width $a_L$ (boxed). Note that while entropy analysis for the shallow water equations typically requires special steps to account for non-constant bathymetry \cite{fjordholm2011well, wintermeyer2018entropy, wu2021high}, the non-symmetric Tadmor condition accounts automatically for the presence of non-constant bathymetry $b(x)$. 
 
Recall that the \rnote{components $v_1, v_2$ of the entropy variables $\bm{v} = [v_1, v_2]$} for the quasi-1D shallow water equations are
\begin{align*}
v_1 =  g(h + b)-\frac{1}{2}u^2, \qquad
v_2 = u,
\end{align*}
with entropy potential $\psi = \frac{1}{2} gahu(h+b)$. We can prove the entropy conservation condition (\ref{eq:nonsym_cond_LR}) by exploiting symmetry of several flux terms
\begin{align*}
\fnt{v}_L^T\bm{f}_{EC}(\fnt{u}_L,\fnt{u}_R) - \fnt{v}_R^T\bm{f}_{EC}(\fnt{u}_R,\fnt{u}_L) 
=& \jump{g(h + b)-\frac{1}{2}u^2}\avg{ahu} + \jump{u}\avg{ahu}\avg{u} \\
&+ u_L \frac{g}{2}a_L h_L(h_R+b_R) - u_R \frac{g}{2}a_R h_R(h_L+b_L).
\end{align*}
Since $\avg{u}\jump{u} = \jump{\frac{1}{2}u^2}$, we have that
\begin{align*}
\jump{g(h + b)-\frac{1}{2}u^2}\avg{ahu} + \jump{u}\avg{ahu}\avg{u} &= \jump{g(h + b)-\frac{1}{2}u^2 + \frac{1}{2}u^2}\avg{ahu} \\
&= \jump{g(h+b)}\avg{ahu}.
\end{align*}
Then, we observe that $\fnt{v}_L^T\bm{f}_{EC}(\fnt{u}_L,\fnt{u}_R) - \fnt{v}_R^T\bm{f}_{EC}(\fnt{u}_R,\fnt{u}_L) $ reduces to
\begin{align*}
\jump{g(h+b)}\avg{ahu} + 
&u_L \frac{g}{2}a_L h_L(h_R+b_R) - u_R \frac{g}{2}a_R h_R(h_L+b_L) \\
=& g\LRp{ \avg{a h u} (h_R + b_R) - \avg{ahu} (h_L + b_L)}\\
+& g\LRp{ \frac{a_L h_L u_L}{2} (h_R + b_R) - \frac{a_R h_R u_R}{2} (h_L + b_L)}\\
=& \frac{g}{2}\LRp{ a_R h_R u_R (h_R + b_R) - a_Lh_Lu_L (h_L + b_L)} = \psi(\bm{u}_L) - \psi(\bm{u}_R).
\end{align*}

\subsection{The quasi-1D compressible Euler equations}

\rnote{As with the shallow water equations, we will generalize an existing entropy conservative flux for the compressible Euler equations to the quasi-1D setting. There are several fluxes which satisfy an entropy conservation condition \cite{ismail2009affordable, chandrashekar2013kinetic, ranocha2018comparison}; we will focus on the numerical fluxes from Ranocha \cite{ranocha2018comparison}. In addition to being entropy conservative, these fluxes are the only entropy conservative numerical flux which (for constant velocities) is also kinetic energy preserving, pressure equilibrium preserving, and has a mass flux which is pressure-independent \cite{ranocha2021preventing}.

To account for spatially varying width $a(x)$ in the quasi-1D Euler equations, we introduce a quasi-1D generalization of Ranocha's fluxes in \cite{ranocha2018comparison} as follows:} 
\begin{align}
    f_{\rho} &= \avg{\rho}_{\log}\avg{au}, \nonumber \\
    f_{\rho u} &= \avg{\rho}_{\log}\avg{au}\avg{u} + \boxed{a_L\avg{p}}, \label{eq:Euler_quasi_1D_flux}\\
    f_{E} &= \frac{1}{2}\avg{\rho}_{\log}\avg{au}\prodmean{u \cdot u} + \frac{1}{\gamma -1}\avg{\rho}_{\log}\avg{ \rho / p }^{-1}_{\log}\avg{au}+\prodmean{p \cdot au},  \nonumber   
\end{align}
with logarithmic and product means
\[
\avg{\rho}_{\log} \coloneqq \frac{\jump{\rho}}{\jump{\log \thickspace \rho}} = \frac{\rho_L-\rho_R}{\log(\rho_L)-\log (\rho_R)}, \qquad \prodmean{u \cdot v} \coloneqq \frac{u_L v_R + u_R v_L}{2}.
\]
Again, the non-symmetric part is boxed for ease of identification. For all numerical experiments, we evaluate the logarithmic mean using the numerically stable method of Ismail and Roe \cite{ismail2009affordable} as implemented in Trixi.jl \cite{schlottkelakemper2021purely, ranocha2021efficient, ranocha2022adaptive}. Like the numerical fluxes for the quasi-1D shallow water equations, the numerical fluxes for the quasi-1D compressible Euler equations are consistent but not symmetric. 

Recall that the \rnote{components $v_1, v_2, v_3$ of the} entropy variables $\bm{v}$ for the quasi-1D compressible Euler equations are
\begin{align*}
v_1 = \frac{\gamma - s}{\gamma -1} - \frac{\rho u^2}{2p}, \qquad 
v_2 =  \frac{\rho u}{p}, \qquad
v_3 =  \frac{-\rho}{p}.
\end{align*}
where $s = \log\LRp{\frac{p}{\rho^\gamma}}$. We now show how to prove the non-symmetric Tadmor condition (\ref{eq:nonsym_cond_LR}). 

First, note that for symmetric flux terms, (\ref{eq:nonsym_cond_LR}) reduces to the standard Tadmor condition involving the jump of the entropy variables. We thus expand out $\fnt{v}_L^T\bm{f}_{EC}(\fnt{u}_L,\fnt{u}_R) - \fnt{v}_R^T\bm{f}_{EC}(\fnt{u}_R,\fnt{u}_L)$ into several terms:
\begin{align}
\fnt{v}_L^T\bm{f}_{EC}(\fnt{u}_L,\fnt{u}_R) &- \fnt{v}_R^T\bm{f}_{EC}(\fnt{u}_R,\fnt{u}_L) = \nonumber\\
&f_\rho \jump{\frac{\gamma - s}{\gamma -1} - \frac{\rho u^2}{2p}} \label{eq:term1}\\
+&f_\rho\avg{u}\jump{\frac{\rho u}{p}} \label{eq:term2}\\ 
+&a_L\avg{p}\frac{\rho_L u_L}{p_L} - a_R\avg{p}\frac{\rho_R u_R}{p_R} \label{eq:term3}\\
+&\frac{1}{2}f_\rho u_L u_R \jump{-\frac{\rho}{p}} \label{eq:term4}\\
+&\frac{1}{\gamma-1}f_\rho \avg{\frac{\rho}{p}}_{\log}^{-1} \jump{-\frac{\rho}{p}} \label{eq:term5}\\
+&\frac{1}{2}\LRp{p_L a_R u_R + p_R a_L u_L} \jump{-\frac{\rho}{p}} \label{eq:term6},
\end{align}
where we have introduced $f_\rho = \avg{\rho}_{\log} \avg{au}$ for brevity.

First, consider the sum of (\ref{eq:term1}), (\ref{eq:term2}), and (\ref{eq:term4}). Straightforward computations show that the term $\avg{u}\jump{\frac{\rho u}{p}}$ in (\ref{eq:term2}) expands out to
\[
\avg{u}\jump{\frac{\rho u}{p}} = \jump{\frac{\rho u^2}{2p}} + \frac{1}{2} \jump{\frac{\rho}{p}} u_L u_R.
\]
Thus, adding (\ref{eq:term1}), (\ref{eq:term2}), and (\ref{eq:term4}) together yields
\[
f_\rho \LRp{\jump{\frac{\gamma - s}{\gamma -1} - \frac{\rho u^2}{2p}}  + \jump{\frac{\rho u^2}{2p}} + \frac{1}{2} \jump{\frac{\rho}{p}} u_L u_R + \frac{1}{2}u_L u_R \jump{-\frac{\rho}{p}}}  = f_\rho {\jump{\frac{\gamma - s}{\gamma -1}}} = \frac{-1}{{\gamma -1}} f_\rho {\jump{s}}.
\]
where we have used that $\gamma$ is constant in the last step. 
Next, we consider (\ref{eq:term5}) $\frac{1}{\gamma-1}f_\rho \avg{\frac{\rho}{p}}_{\log}^{-1} \jump{-\frac{\rho}{p}}$. Note that
\[
\avg{\frac{\rho}{p}}_{\log}^{-1} \jump{-\frac{\rho}{p}} = -\frac{\jump{\log\LRp{\frac{\rho}{p}}}}{\jump{\frac{\rho}{p}}}\jump{\frac{\rho}{p}} = -\jump{\log\LRp{\frac{\rho}{p}}}.
\]
Then, (\ref{eq:term5}) reduces to 
\[
\frac{1}{\gamma-1}f_\rho \avg{\frac{\rho}{p}}_{\log}^{-1} \jump{-\frac{\rho}{p}} = \frac{-1}{\gamma - 1} f_\rho \jump{\log\LRp{\frac{\rho}{p}}}.
\]
Summing (\ref{eq:term1}), (\ref{eq:term2}), (\ref{eq:term4}), and (\ref{eq:term5}) and using that $s = \log\LRp{\frac{p}{\rho^\gamma}}$ then yields
\begin{align}
\frac{-1}{\gamma -1} f_\rho \jump{\log\LRp{\frac{p}{\rho^\gamma}} + \log\LRp{\frac{\rho}{p}}} &= \frac{-1}{\gamma -1} f_\rho \jump{\log\LRp{p} - \gamma \log\LRp{\rho} + \log\LRp{\rho} - \log\LRp{p}} \nonumber\\
&= \frac{-1}{\gamma -1} f_\rho \jump{ (1 - \gamma) \log\LRp{\rho}} \nonumber\\
&= f_\rho \jump{ \log\LRp{\rho}} = \avg{au}\jump{\rho}, \label{eq:term1245}
\end{align}
where we have used the definition of $f_\rho = \avg{\rho}_{\log} \avg{au} = {\jump{\rho}}/{\jump{\log(\rho)}}$ in the final step.

We now simplify (\ref{eq:term1245}), (\ref{eq:term3}), and (\ref{eq:term6}). The term (\ref{eq:term3}) can be simplified to
\[
a_L\avg{p}\frac{\rho_L u_L}{p_L} - a_R\avg{p}\frac{\rho_R u_R}{p_R} = \frac{1}{2} \LRp{\jump{a \rho u} + a_L \rho_L u_L \frac{p_R}{p_L} - a_R \rho_R u_R \frac{p_L}{p_R}}.
\]
Similarly, we can simplify (\ref{eq:term6}) to 
\[
\frac{1}{2}\LRp{p_L a_R u_R + p_R a_L u_L} \jump{-\frac{\rho}{p}} = \frac{1}{2} \LRp{-a_L \rho_L u_L \frac{p_R}{p_L}+ a_R \rho_R u_R \frac{p_L}{p_R} -\rho_L a_R u_R + \rho_R a_L u_L}.
\]
Finally, we can simplify (\ref{eq:term1245}) to 
\[
\avg{au}\jump{\rho} = \frac{1}{2}(a_L u_L + a_R u_R)(\rho_L - \rho_R) = \frac{1}{2}\LRp{\jump{a\rho u} + \rho_L a_R u_R - \rho_R a_L u_L}.
\]
Summing the simplified versions of (\ref{eq:term1245}), (\ref{eq:term3}), and (\ref{eq:term6}) together yields that $\fnt{v}_L^T\bm{f}_{EC}(\fnt{u}_L,\fnt{u}_R) - \fnt{v}_R^T\bm{f}_{EC}(\fnt{u}_R,\fnt{u}_L) = \jump{a\rho u} = \jump{\psi(\bm{u})}$.

\section{Entropy stable flux differencing schemes for quasi-1D systems}

We wish to derive entropy conservative fluxes for the quasi-1D shallow water equations and for the quasi-1D compressible Euler equations. We will then construct numerical fluxes and perform an analysis based on an algebraic entropy stable formulation. 

\subsection{Notation}

For the remainder of the paper, we use \rnote{that $\circ $ denotes the Hadamard product (e.g., the element-wise product) of two vectors or matrices with the same dimensions.}

\rnote{We will also use $\avg{\fnt{a}} = \frac{1}{2}\LRp{\fnt{a}_L + \fnt{a}_R}$ to denote the arithmetic mean between two states $\fnt{a}_L, \fnt{a}_R$. This will also be used to denote the average of the solution at the $i$th node and the solution at the $j$th node, e.g., $\avg{\fnt{a}} = \frac{1}{2}(\fnt{a}_i + \fnt{a}_j)$, where $i$ and $j$ are indices which appear in the context of proofs of entropy conservation. We will also use the jump $\jump{a} = a_L - a_R$ to denote the difference between two states $a_L, a_R$.}

\subsection{Algebraic formulation in terms of SBP operators}

We analyze a one-dimensional flux differencing formulation written in terms of a summation-by-parts (SBP) differentiation matrix $\fnt{Q}$, a boundary matrix $\fnt{B}$, and a diagonal reference mass matrix $\fnt{M}$ on a reference element $[-1,1]$. These matrices satisfy the following properties
\[
\fnt{M}_{ii} > 0, \qquad \fnt{Q}+\fnt{Q}^T = \fnt{B}, \qquad \fnt{Q}\fnt{1} = \fnt{0}.
\]
Suppose our domain $\Omega$ is now decomposed into multiple interval subdomains $D^k$, each of which has some size $\LRb{D^k} = h_k$. Transforming the PDE to the reference interval allows us to define a local formulation (similar to local formulations for discontinuous Galerkin (DG) methods \cite{hesthaven2007nodal}) over each element:
\begin{align}
\fnt{M}_k\td{\fnt{u}}{t} + \LRp{\LRp{\fnt{Q}-\fnt{Q}^T} \circ \fnt{F}}\fnt{1} + \fnt{B}\bm{f}^*(\fnt{u}, \fnt{u}^+) = \fnt{0}, \qquad \fnt{F}_{ij} = \bm{f}_{EC}(\bm{u}_i, \bm{u}_j).
\label{eq:global_DG_formulation_dis}
\end{align}
where $\fnt{u}^+$ denotes the ``exterior'' solution state at an element interface or domain boundary, and can be used to weakly enforce either continuity between elements or appropriate boundary conditions. 
We have also introduced the physical mass matrix $\fnt{M}_k = h_k\fnt{M}$, as well as the interface numerical flux $\bm{f}^*$ (to be specified later) and the flux matrix $\fnt{F}$ whose entries correspond to evaluations of the numerical flux $\bm{f}_{EC}(\bm{u}_L, \bm{u}_R)$. 

Note that since the flux is non-symmetric, the order of the arguments for the boundary flux is important. This general form will be used to analyze the entropy stability of multi-domain summation-by-parts finite difference schemes, as well as discontinuous Galerkin methods. 

\begin{remark}
Using the SBP property, we can rewrite (\ref{eq:global_DG_formulation_dis}) in ``strong form'' 
\begin{align}
\fnt{M}_k\td{\fnt{u}}{t} + 2\LRp{\fnt{Q} \circ \fnt{F}}\fnt{1} + \fnt{B}\LRp{\bm{f}_{EC}(\fnt{u}, \fnt{u}^+) - \bm{f}_{EC}(\fnt{u}, \fnt{u})} = \fnt{0}.
\label{eq:global_DG_formulation_strong}
\end{align}
This form is more convenient for analyzing conservation and high order accuracy.
\end{remark}

\subsection{Semi-discrete entropy analysis}

\rnote{We now show how to derive a semi-discrete entropy inequality from (\ref{eq:global_DG_formulation_dis}). We note that this derivation assumes positivity of appropriate variables (e.g., $h$ for the quasi-1D shallow water equations, density and pressure for the quasi-1D compressible Euler equations. We also note that the formulations presented in this work do not guarantee positivity of such variables, and appropriate fully-discrete limiting techniques must be used to enforce positivity \cite{qian2018positivity, zhang2017positivity, DZANIC2022111501, lin2023positivity}. These techniques fall outside the scope of this paper, and will be considered in future work.}

If we multiply Eq.\ (\ref{eq:global_DG_formulation_dis}) by the entropy variables $\bm{v}^T$, we have 
\begin{align}
\fnt{v}^T\fnt{M}_k\td{\fnt{u}}{t} + \fnt{v}^T\LRp{\LRp{\fnt{Q}-\fnt{Q}^T}\circ \fnt{F}}\fnt{1} + \fnt{v}^T\fnt{B}\bm{f}^*(\fnt{u}, \fnt{u}^+) = \fnt{0}.
\label{eq:global_DG_entropy_dis}
\end{align}
Because $\fnt{M}_k$ is diagonal, we have that 
\[
\fnt{v}^T\fnt{M}_k\td{\fnt{u}}{t} = \sum_i \fnt{v}_i \fnt{M}_{k, ii} \td{\fnt{u}_i}{t} = \sum_i \fnt{M}_{k, ii} \LRl{\pd{S}{\bm{u}}}_{\fnt{v}_i}  \td{\fnt{u}_i}{t} = \sum_i \fnt{M}_{k, ii} \td{S(\fnt{u}_i)}{t} = \fnt{1}^T\fnt{M}_k\td{S(\fnt{u})}{t}.
\]
Thus, we have that a scheme is entropy conservative such that $\fnt{1}^T\fnt{M}\td{S(\fnt{u})}{t} = 0$ if 
\begin{align}
\fnt{v}^T\LRp{2\fnt{Q}\circ \fnt{F}}\fnt{1}  = \fnt{0}.
\label{eq:global_DG_entropy_dis_rhs}
\end{align}
Typical proofs of entropy conservation assume that $\bm{f}_{EC}$ is symmetric \cite{chan2018discretely}. However, these approaches cannot be directly applied to the quasi-1D equations because they cannot be written in conservative form (due to the presence of non-symmetric terms involving $a(x)$). As a result, the numerical fluxes $\bm{f}_{EC}(\bm{u}_L, \bm{u}_R)$ are no longer symmetric. However, the proof of entropy conservation can be modified to account for asymmetry in the flux. Instead, we can derive that 
\begin{align*}
\fnt{v}^T\LRp{\LRp{\fnt{Q} - \fnt{Q}^T}\circ \fnt{F}}\fnt{1} &= \sum_{i,j} \LRp{\fnt{Q}_{ij} - \fnt{Q}_{ji}}\fnt{v}_i^T\bm{f}_{EC}(\fnt{u}_i,\fnt{u}_j) \\
&= \sum_{i,j} \fnt{Q}_{ij}\fnt{v}_i^T\bm{f}_{EC}(\fnt{u}_i,\fnt{u}_j) - \sum_{i,j} \fnt{Q}_{ji}\fnt{v}_i^T\bm{f}_{EC}(\fnt{u}_i,\fnt{u}_j).
\end{align*}
Exchanging indices in the second sum allows us to simplify this expression to 
\begin{equation}
\fnt{v}^T\LRp{\LRp{\fnt{Q}-\fnt{Q}^T} \circ \fnt{F}}\fnt{1} = \sum_{i,j} \fnt{Q}_{ij}\LRp{\fnt{v}_i^T\bm{f}_{EC}(\fnt{u}_i,\fnt{u}_j) - \fnt{v}_j^T\bm{f}_{EC}(\fnt{u}_j,\fnt{u}_i)}.
\label{eq:entropy_residual}
\end{equation}
If the non-symmetric entropy conservation condition (\ref{eq:nonsym_cond_LR}) holds, then
\begin{equation}
\fnt{v}_i^T\bm{f}_{EC}(\fnt{u}_i,\fnt{u}_j) - \fnt{v}_j^T\bm{f}_{EC}(\fnt{u}_j,\fnt{u}_i) = \psi(\fnt{u}_i) - \psi(\fnt{u}_j)
\label{eq:nonsym_cond}
\end{equation}
and (\ref{eq:entropy_residual}) reduces to
\begin{align*}
\fnt{v}^T\LRp{\LRp{\fnt{Q}-\fnt{Q}^T}\circ \fnt{F}}\fnt{1} &= \sum_{i,j} \fnt{Q}_{ij}\LRp{\fnt{v}_i^T\bm{f}_{EC}(\fnt{u}_i,\fnt{u}_j) - \fnt{v}_j^T\bm{f}_{EC}(\fnt{u}_j,\fnt{u}_i)} \\
&= \sum_{i,j} \fnt{Q}_{ij} \LRp{\psi(\fnt{u}_i) - \psi(\fnt{u}_j)} \\
&= \psi(\fnt{u})^T\fnt{Q}\fnt{1} - \fnt{1}^T\fnt{Q}\psi(\fnt{u}) = - \fnt{1}^T\fnt{B}\psi(\fnt{u}).
\end{align*}
where we have used that $\fnt{Q}\fnt{1} = \fnt{0}$ and the SBP property in the final step. We summarize this in the following theorem:
\begin{theorem}
\label{thm:ec_local}
Assume that the entropy conservative flux $\bm{f}_{EC}$ satisfies the non-symmetric Tadmor condition (\ref{eq:nonsym_cond_LR}). Then, (\ref{eq:global_DG_formulation_dis}) satisfies the following local statement of entropy conservation:
\[
\fnt{1}^T \fnt{M}_k \td{S(\fnt{u})}{t} + \fnt{1}^T\fnt{B}\LRp{\fnt{v}^T\bm{f}^*(\fnt{u},\fnt{u}^+) - \psi(\fnt{u})} = 0.
\]
\end{theorem}

We now treat interface terms, which involve the outward normals as encoded by the matrix $\fnt{B}$. Suppose that $\fnt{f}^*$ is entropy stable, and consider a shared face between two elements $D^k$ and $D^{k,+}$. The outward normal for $D^{k,+}$ is the negation of the outward normal for $D^k$, so summing face contributions in Theorem~\ref{thm:ec_local} and using (\ref{eq:nonsym_cond_LR_ES}) yields
\[
\fnt{v}^T\bm{f}^*(\fnt{u}, \fnt{u}^+) - \LRp{\fnt{v}^+}^T\bm{f}^*(\fnt{u}^+, \fnt{u}) \geq \psi(\fnt{u}) - \psi(\fnt{u}^+).
\]
Summing up interface contributions over all elements $D^k$, we observe that all instances of $\psi(\fnt{u})$ cancel with each other. The only terms which remain after summation over all elements are the terms corresponding to the global domain boundaries. We summarize this in the following theorem:
\begin{theorem}
\label{thm:ec_global}
If the entropy conservative flux $\bm{f}_{EC}$ is entropy conservative in the sense of (\ref{eq:nonsym_cond_LR}) and $\bm{f}^*$ is entropy stable in the sense of (\ref{eq:nonsym_cond_LR_ES}), (\ref{eq:global_DG_formulation_dis}) satisfies the following statement of global entropy dissipation:
\begin{equation}
\LRp{\sum_k \fnt{1}^T \fnt{M}_k \td{S(\fnt{u})}{t}} + \LRp{\fnt{v}_R^T\bm{f}^*(\fnt{u}_R,\fnt{u}_R^+) - \psi(\fnt{u}_R)} - \LRp{\fnt{v}_L^T\bm{f}^*(\fnt{u}_L,\fnt{u}_L^+) - \psi(\fnt{u}_L)} \leq 0,
\label{eq:global_entropy_dissip_noBCs}
\end{equation}
where $\fnt{v}_L, \fnt{u}_L$ and $\fnt{v}_R, \fnt{u}_R$ denote solution states at the left and right endpoint of the domain $\Omega$, and $\fnt{u}_{L}^+, \fnt{u}_R^+$ denote exterior states used to impose boundary conditions. 
\end{theorem}

Adding an entropy dissipative interface penalization or entropy dissipative physical diffusion term produces an entropy stable discretization \cite{chan2022efficient}. For example, an entropy dissipative scheme can be constructed by adding a local Lax-Friedrichs penalty to the entropy conservative flux at interfaces
\begin{equation}
\bm{f}^*(\fnt{u}, \fnt{u}^+) = \bm{f}_{EC}(\fnt{u}, \fnt{u}^+) - \frac{\lambda}{2}\jump{\fnt{u}}n, \qquad \lambda > 0
\label{eq:ec_plus_LxF}
\end{equation}
where we have introduced the jump $\jump{\fnt{u}} = \fnt{u}^+ - \fnt{u}$ and the outward normal $n \pm 1$. Similarly, one can also use a local Lax-Friedrichs flux, which is entropy stable if $\lambda$ is sufficiently large \cite{warnecke2004solution, kroner2008minimum}.

\begin{remark}
The non-symmetric condition (\ref{eq:nonsym_cond_LR}) results in a simpler semi-discrete entropy analysis for the shallow water equations. For example, \cite{wintermeyer2017entropy, wu2021high} treat the bathymetry terms separately from other symmetric flux terms, and \cite{fjordholm2011well} modifies the symmetric flux condition to explicitly account for discontinuities in bathymetry. In contrast, the condition (\ref{eq:nonsym_cond_LR}) allows the bathymetry terms to be absorbed naturally into the definition of the flux.
\end{remark}

\subsubsection{Wall boundary conditions}

Finally, we discuss boundary conditions for which we can prove global entropy dissipation. For periodic boundary conditions, $\fnt{u}_L^+ = \fnt{u}_R$ and $\fnt{u}_R^+ = \fnt{u}_L$ in (\ref{eq:global_entropy_dissip_noBCs}), and Theorem~\ref{thm:ec_global} and (\ref{eq:nonsym_cond_LR_ES}) imply a global statement of entropy dissipation:
\[
\sum_k \fnt{1}^T \fnt{M}_k \td{S(\fnt{u})}{t} \leq 0.
\]

From here onwards, we focus on reflective wall boundary conditions (e.g.,  the normal velocity is zero). For the analysis of boundary conditions, we now restrict ourselves to the entropy stable flux (\ref{eq:ec_plus_LxF}) constructed using a local Lax-Friedrichs penalization. For the quasi-1D shallow water equations, we choose the exterior state $\fnt{u}^+$ appropriately. For the quasi-1D shallow water equations, we set $\fnt{u}^+ = \LRc{ah, -ahu}$. Under these ``mirror state'' boundary conditions, $\avg{ahu} = 0$, so $\fnt{v}^T\bm{f}^*(\fnt{u}, \fnt{u}^+) - \psi(\fnt{u})$ reduces to
\[
\fnt{v}^T\bm{f}^*(\bm{u}, \bm{u}^+) - \psi(\fnt{u}) = \frac{g}{2}a h(h + b) u + \lambda ahu^2 - \frac{1}{2} gah(h+b)u = \lambda ahu^2.
\]
 Theorem~\ref{thm:ec_global} then implies entropy stability if $\lambda ahu^2 > 0$, which holds if $ah > 0$ and $\lambda > 0$. 

For the quasi-1D compressible Euler equations we impose reflective wall conditions through the mirror states $\fnt{u}^+ = \LRc{a\rho, -a\rho u, aE}$. Under this assumption, $p = p^+$ and the boundary contributions in (\ref{eq:global_entropy_dissip_noBCs}) reduce to
\[
\fnt{v}^T\bm{f}^*(\bm{u}, \bm{u}^+) - \psi(\fnt{u}) = \begin{bmatrix}v_1\\ v_2 \\ v_3 \end{bmatrix}^T \LRp{\begin{bmatrix}0 \\ a p \\ 0\end{bmatrix} - 
\frac{\lambda}{2}\begin{bmatrix} 0 \\ -2 a \rho u \\ 0 \end{bmatrix}} - a\rho u,
\]
where we have used that $\prodmean{p \cdot au} = \frac{1}{2}\LRp{p a^+ u^+ + p^+ a u} = \frac{1}{2}\LRp{-p a u + p a u} = 0$. Since $v_2 = \frac{\rho u}{p}$, this expression reduces to 
\[
\fnt{v}^T\bm{f}^*(\bm{u}, \bm{u}^+) - \psi(\fnt{u}) = \lambda a \rho u \frac{\rho u}{p} =   \lambda a \LRp{\rho u}^2 \frac{1}{p},
\]
Theorem~\ref{thm:ec_global} implies entropy stability if $\lambda a \LRp{\rho u}^2 \frac{1}{p}$ is non-negative, which holds if $a > 0$ and $p > 0$. 


\subsection{Conservation}

Because quasi-1D systems are not in conservation form, the usual proofs of conservation do not hold. However, we can still show conservation of mass for both quasi-1D shallow water and compressible Euler, and we can show conservation of energy for the quasi-1D compressible Euler equations. Moreover, the rate of change of the mean momentum mimics the continuous case. 

Semi-discrete conservation is derived by \rnote{multiplying by $\fnt{1}^T$} and summing over all elements $D^k$
\[
\sum_k \fnt{1}^T\fnt{M}_k\td{\fnt{u}}{t} + \fnt{1}^T\LRp{\LRp{\fnt{Q}-\fnt{Q}^T} \circ \fnt{F}}\fnt{1} + \fnt{1}^T\fnt{B}\bm{f}^*(\fnt{u}, \fnt{u}^+) = 0.
\]
Let $\bm{f}_{EC}(\bm{u}_L, \bm{u}_R) = \bm{f}_{\rm sym}(\bm{u}_L, \bm{u}_R) + \bm{f}_{\rm nonsym}(\bm{u}_L, \bm{u}_R)$, where $\bm{f}_{\rm sym}(\bm{u}_L, \bm{u}_R)$ denotes the symmetric part of the flux and $\bm{f}_{\rm nonsym}(\bm{u}_L, \bm{u}_R)$ denotes the non-symmetric part. Then, $\fnt{F} = \fnt{F}_{\rm sym} + \fnt{F}_{\rm nonsym}$, where $\fnt{F}_{\rm sym}, \fnt{F}_{\rm nonsym}$ denote flux matrices constructed $\bm{f}_{\rm sym}$ and $\bm{f}_{\rm nonsym}$, respectively. Then, because $\LRp{\fnt{Q}-\fnt{Q}^T} \circ \fnt{F}_{\rm sym}$ is a skew-symmetric matrix, we can simplify the conservation expression to
\begin{align*}
\fnt{1}^T\LRp{\LRp{\fnt{Q}-\fnt{Q}^T} \circ \fnt{F}}\fnt{1} &= \fnt{1}^T\LRp{\LRp{\fnt{Q}-\fnt{Q}^T} \circ \fnt{F}_{\rm nonsym}}\fnt{1} \\
&= \fnt{1}^T\LRp{2\fnt{Q} \circ \fnt{F}_{\rm nonsym}}\fnt{1} - \fnt{1}^T\LRp{\fnt{B} \circ \fnt{F}_{\rm nonsym}}\fnt{1} \\
&= \fnt{1}^T\LRp{2\fnt{Q} \circ \fnt{F}_{\rm nonsym}}\fnt{1} - \fnt{1}^T\fnt{B} \bm{f}_{\rm nonsym}(\fnt{u},\fnt{u}). 
\end{align*}
where we have used the SBP property $\fnt{Q} = \fnt{B} - \fnt{Q}^T$ and the fact that $\fnt{B}$ is diagonal in the last two steps. Thus, we have that
\begin{equation}
\fnt{1}^T\fnt{M}_k\td{\fnt{u}}{t} + \fnt{1}^T\LRp{2\fnt{Q} \circ \fnt{F}_{\rm nonsym}}\fnt{1} + \fnt{1}^T\fnt{B}\LRp{\bm{f}^*(\fnt{u}, \fnt{u}^+) - \bm{f}_{\rm nonsym}(\fnt{u}, \fnt{u})} = 0.
\label{eq:conservation}
\end{equation}
Because the entropy conservative fluxes (\ref{eq:SWE_quasi_1D_flux}) and (\ref{eq:Euler_quasi_1D_flux}) are symmetric for the mass equation and (for the quasi-1D compressible Euler case) the energy equation, the components of $\bm{f}_{\rm nonsym}(\fnt{u}, \fnt{u})$ corresponding to the mass and (if applicable) the energy equations are zero. This implies the usual conservation condition (e.g., (77) in \cite{chan2018discretely}) for mass and energy. 

For the momentum equations, we seek to mimic the continuous conservation condition. Integrating the momentum equation of the quasi-1D shallow water system (\ref{eq:SWE_quasi_1D}) yields
\[
\int_{\Omega} \pd{ah u}{t} + \int_{\Omega} gah \pd{}{x}\LRp{h+b} + \int_{\partial \Omega} ahu^2 n = 0
\]
where $n = \pm 1$ denotes the outward normal on the domain boundary $\partial \Omega$. We will show that the discrete statement of conservation (\ref{eq:conservation}) mimics these conservation conditions. We begin by simplifying the volume term $\fnt{1}^T\LRp{2\fnt{Q} \circ \fnt{F}_{\rm nonsym}}\fnt{1}$. Note that for the quasi-1D shallow water equations, the non-symmetric part of the flux corresponds to $\frac{1}{2}ga_L h_L(h_R + b_R)$, such that
\[
\fnt{1}^T\LRp{2\fnt{Q} \circ \fnt{F}_{\rm nonsym}}\fnt{1} = \fnt{1}^T g\fnt{a} \circ \fnt{h} \circ \LRp{\fnt{Q} \LRp{\fnt{h} + \fnt{b}}} \approx \int_{D^k} g ah \pd{}{x}\LRp{h+b}.
\]
For the quasi-1D compressible Euler equations (\ref{eq:Euler_q1D}), integrating the momentum equation yields
\[
\int_{\Omega} \pd{a\rho u}{t} + \int_{\Omega} a\pd{p}{x} + \int_{\partial \Omega} a\rho u^2 n = 0.
\]
Similarly, the non-symmetric part corresponds to $a_L \avg{p}$, such that
\[
\fnt{1}^T\LRp{2\fnt{Q} \circ \fnt{F}_{\rm nonsym}}\fnt{1} = \fnt{1}^T \fnt{a} \circ \LRp{\fnt{Q} \fnt{p}} \approx \int_{D^k} \bnote{a \pd{p}{x}}.
\]
Finally, we note that 
\begin{equation}
\fnt{1}^T\fnt{B}\LRp{\bm{f}^*(\fnt{u}, \fnt{u}^+) - \bm{f}_{\rm nonsym}(\fnt{u}, \fnt{u})}
\label{eq:cons_boundary}
\end{equation}
is a consistent approximation to the boundary integral of conservative fluxes in each quasi-1D system. For the quasi-1D shallow water equations, (\ref{eq:cons_boundary}) corresponds to boundary contributions of the form 
\begin{equation}
\bm{f}^*(\fnt{u}, \fnt{u}^+) - \bm{f}_{\rm nonsym}(\fnt{u}, \fnt{u}) = \begin{bmatrix} 0\\ \avg{ahu}\avg{u} + \frac{1}{2}g a_Lh_L \jump{h + b} \end{bmatrix} - \frac{\lambda}{2} \jump{\bm{u}}n, 
\label{eq:swe_interface_flux}
\end{equation}
such that $\fnt{1}^T\fnt{B}\LRp{\bm{f}^*(\fnt{u}, \fnt{u}^+) - \bm{f}_{\rm nonsym}(\fnt{u}, \fnt{u})} \approx \int_{\partial \Omega} ahu^2$. 
For the quasi-1D compressible Euler equations, (\ref{eq:cons_boundary}) corresponds to boundary contributions of the form $\avg{\rho}_{\log}\avg{au}\avg{u} + a_L \frac{1}{2}\jump{p} \approx a\rho u^2$. Note that for both sets of equations, the presence of entropy dissipative jump penalization terms in $\bm{f}^*$ do not negatively impact consistency.

\subsection{High order accuracy}

It was shown in \cite{crean2018entropy} that if $\fnt{M}^{-1}\fnt{Q}$ is a high order accurate nodal differentiation operator, that $\fnt{M}^{-1}\LRp{\fnt{Q}\circ\fnt{F}}\fnt{1}$ yields a high order accurate approximation to $\pd{\bm{f}(\bm{u})}{x}$. Unfortunately, the proof of high order accuracy relies on the symmetry of the flux $\bm{f}_{EC}$, and does not hold for a non-symmetric flux. However, the structure of the fluxes $\bm{f}_{EC}$ for the quasi-1D shallow water and compressible Euler equations allows for a straightforward proof of high order consistency. 
\begin{theorem}
\rnote{Assume that $a$ and $h$ (for the quasi-1D shallow water equations) or $\rho, p$ (for the quasi-1D compressible Euler equations) are sufficiently regular and uniformly bounded away from zero. Then,} (\ref{eq:global_DG_formulation_dis}) is an \rnote{$O(h^N)$} accurate approximation of (\ref{eq:SWE_quasi_1D}) and 
(\ref{eq:Euler_q1D}) under periodic or wall boundary conditions.   
\end{theorem}
\begin{proof}
We use the strong form of the discretization (\ref{eq:global_DG_formulation_strong}) to show $O(h^N)$ accuracy. First, we note that the symmetric terms in the entropy conservative flux $\bm{f}_{EC}$ are consistent numerical fluxes for conservative terms in (\ref{eq:SWE_quasi_1D}) and 
(\ref{eq:Euler_q1D}). Since these flux terms are consistent and symmetric, the proof of \rnote{$O(h^N)$} accuracy of the flux differencing approximation in \cite{crean2018entropy} holds for \rnote{these conservative terms} (see also \cite{fjordholm2012arbitrarily, fisher2013discretely}). All that remains is to show that the non-conservative terms in $\bm{f}_{EC}$ induce a high order accurate approximations to the non-conservative terms in (\ref{eq:SWE_quasi_1D}) and (\ref{eq:Euler_q1D}). 

\rnote{Let $I_N$ denote the degree $N$ polynomial interpolation operator, and} let $\fnt{F}_{\rm nc}$ denote the flux matrix corresponding to the scalar non-conservative terms in either the quasi-1D shallow water or compressible Euler equations, and let $\fnt{D} = \fnt{M}^{-1}\fnt{Q}$ denotes the SBP differentiation matrix. For the quasi-1D shallow water equations, the non-conservative terms in (\ref{eq:SWE_quasi_1D}) are $gah\pd{}{x}\LRp{h+b}$. The corresponding non-conservative flux terms are $\frac{1}{2}ga_Lh_L(h_R+b_R)$, and following \cite{gassner2016split}, the flux differencing approximation reduces to 
\begin{align*}
\LRp{\fnt{M}^{-1}\LRp{2\fnt{Q} \circ \fnt{F}_{\rm nc}}\fnt{1}}_i = \sum_j \fnt{M}_{ii}^{-1}\fnt{Q}_{ij} g\fnt{a}_i \fnt{h}_i(\fnt{h}_j+\fnt{b}_j) = g \fnt{a}_i \fnt{h}_i \sum_j \fnt{D}_{ij}  (\fnt{h}_j+\fnt{b}_j).
\end{align*}
This term corresponds to $g \fnt{a} \circ \fnt{h} \circ \LRp{\fnt{D} \LRp{\fnt{h} + \fnt{b}}}$, which \rnote{in turn corresponds to $I_N\LRp{gah\pd{}{x}\LRp{I_N(h+b)}}$. For sufficiently regular $a, h, b$, this interpolant is an $O(h^N)$ approximation to $gah\pd{}{x}\LRp{h+b}$.}

For the quasi-1D compressible Euler equations, the non-conservative term is $a(x)\pd{p}{x}$ and the corresponding flux terms in $\bm{f}_{EC}$ are $a_L \avg{p} = a_L \frac{1}{2}(p_L + p_R)$. The flux differencing contribution for this non-symmetric term is
\[
\LRp{\fnt{M}^{-1}\LRp{2\fnt{Q} \circ \fnt{F}_{\rm nc}}\fnt{1}}_i = \sum_j \fnt{M}_{ii}^{-1} \fnt{Q}_{ij} \fnt{a}_i (\fnt{p}_i+\fnt{p}_j) = \fnt{a}_i \sum_j \fnt{D}_{ij} \fnt{p}_j, 
\]
where we have used that $\sum_j \fnt{Q}_{ij} = 0$ for any first order accurate differentiation matrix \cite{gassner2016split}. \rnote{This corresponds to $\fnt{a}\circ \LRp{\fnt{D}\fnt{p}}$, which in turn corresponds to $I_N\LRp{a\pd{I_N(p)}{x}}$. For sufficiently regular $a, p$, this interpolant is an $O(h^{N})$ approximation to $a(x)\pd{p}{x}$.

Finally, uniform boundedness away from zero implies that $a \geq a_0 > 0$ and $h \geq h_0 > 0$ (for the quasi-1D shallow water equations) or $\rho, p \geq \rho_0, p_0 > 0$ (for the quasi-1D compressible Euler equations). From the expressions (\ref{eq:SWE_quasi_1D_flux}), (\ref{eq:Euler_quasi_1D_flux}), uniform boundedness away from zero implies that $\bm{f}_{EC}$ is a uniformly continuous function such that $\nor{\bm{f}_{EC}(\fnt{u}, \fnt{u}^+) - \bm{f}_{EC}(\fnt{u}, \fnt{u})}$ is proportional to $\nor{\fnt{u} - \fnt{u}^+}$.}
\end{proof}

\subsection{Well-balancedness for the quasi-1D shallow water equations}

Next, we show that our numerical scheme (\ref{eq:global_DG_formulation_dis}) is well-balanced for the quasi-1D shallow water equations. The ``lake-at-rest'' well-balanced property preserves steady states where $u=0$ and $h + b=c$, where $c$ is some constant. 

We first analyze the case of continuous bathymetry \bnote{under the} flux (\ref{eq:ec_plus_LxF}). Since $h+b$ is constant, continuity of $b$ implies continuity of $h$, such that the Lax-Friedrichs penalization terms vanish in (\ref{eq:swe_interface_flux}). Since $\fnt{u} = \fnt{0}$, the entropy conservative part of (\ref{eq:swe_interface_flux}) vanishes as well. We now show that the volume terms also vanish. Since the semi-discrete solution satisfies $\fnt{ahu} = \fnt{0}$, we immediately have that the flux for the $ah$ equation vanishes and thus $\pd{}{t}ah = 0$. 

It remains to show that $(\fnt{Q}\circ \fnt{F}_{ahu})\fnt{1} = \fnt{0}$, which would imply $\pd{}{t}ahu = 0$. Let $\fnt{F}_{ahu}$ denote the flux matrix constructed from the flux for the $ahu$ equation. Since $\avg{\fnt{ahu}}\avg{\fnt{u}} = 0$, we have that
\begin{align}
\nonumber
\LRp{(\fnt{Q}\circ \fnt{F}_{ahu})\fnt{1}})_i =& \sum_j \fnt{Q}_{ij}\LRp{\frac{1}{2}g\fnt{a}_i\fnt{h}_i\fnt{h}_j+\frac{1}{2}g\fnt{a}_i\fnt{h}_i\fnt{b}_j} \\
=&\frac{g\fnt{a}_i\fnt{h}_i}{2}\sum_j \fnt{Q}_{ij}(\fnt{h}_j+\fnt{b}_j) =\frac{g\fnt{a}_i\fnt{h}_i}{2}\sum_j \fnt{Q}_{ij} c = 0.
\end{align}
This implies that our numerical scheme is well-balanced for the quasi-1D shallow water equations. 

For discontinuous bathymetry profiles, $\jump{h}$ no longer vanishes and the scheme (\ref{eq:global_DG_formulation_dis}) with local Lax-Friedrichs flux penalization (\ref{eq:ec_plus_LxF}) is no longer well-balanced. However, as noted in \cite{fjordholm2011well, wintermeyer2017entropy}, if the flux penalization is defined in terms of the entropy variables, then because $\jump{v_1} = \jump{g(h + b) - \frac{1}{2}u^2} = 0$ and $\jump{v_2} = \jump{u} = 0$ for $h+b = c$ and $u = 0$, the interface term vanishes. Thus, our scheme  (\ref{eq:global_DG_formulation_dis}) is well-balanced for discontinuous bathymetry so long as the interface flux is of the form
\begin{equation}
\bm{f}^*(\bm{u}^+, \bm{u}) = \bm{f}_{EC}(\bm{u}^+, \bm{u}) - \bm{R}\jump{\bm{v}(\bm{u})} n
\label{eq:wb_lxf}
\end{equation}
where $\bm{v}(\bm{u})$ denote the entropy variables and $\bm{R}$ is a positive-definite matrix which is single-valued across the interface. In this work, we utilize the matrix
\[
\bm{R} = \frac{\lambda}{2} \LRl{\pd{\bm{u}}{\bm{v}}}_{\avg{\bm{u}}}, \qquad \pd{\bm{u}}{\bm{v}} =  \frac{1}{ah} \begin{bmatrix}
(g h + u^2)  &   -u\\
-u  		  &     1
\end{bmatrix}
\]
which is the Jacobian matrix of the transformation between conservative and entropy variables, evaluated using the average states at an interface. The flux penalty $\bm{R}\jump{\bm{v}(\bm{u})}$ then corresponds to a Lax-Friedrichs flux penalization expressed using entropy variables. We note that the dissipation matrices of \cite{fjordholm2011well, wintermeyer2017entropy} will also preserve the lake-at-rest steady state. 

\begin{remark}
There also exist well-balanced first and second order finite volume schemes for the quasi-1D compressible Euler equations \cite{kroner2005numerical, helluy2012well}; however, the high order schemes presented in this paper do not appear to preserve such steady states. 
\end{remark}

\section{Numerical experiments} 
\label{sec:num}

In this section, we verify the entropy conservation/stability and high order accuracy of the formulation (\ref{eq:global_DG_formulation_dis}) constructed using the entropy conservative fluxes for the shallow water equations (\ref{eq:SWE_quasi_1D_flux}) and the compressible Euler equations (\ref{eq:Euler_quasi_1D_flux}). We focus on high order discontinuous Galerkin spectral element method (DGSEM) discretizations, for which the SBP matrices in (\ref{eq:global_DG_formulation_dis}) are constructed over each element using $(N+1)$-point Lobatto quadrature points \cite{gassner2013skew}. 

For all convergence tests, we report the total $L^2$ error 
\[
\nor{\bm{u}}_{L^2}^2 = \sum_i^{N_{\rm vars}} \nor{\frac{\bm{u}_i}{a}}^2_{L^2}.
\]
where $N_{\rm vars}$ denotes the number of conservative variables in the system. We divide by $a(x)$ to recover the error in the ``standard'' conservative variables (e.g., $h, hu$ for shallow water and $\rho, \rho u, E$ for compressible Euler).

All numerical experiments are implemented using the Julia programming language \cite{bezanson2017julia}, the OrdinaryDiffEq.jl package \cite{rackauckas2017differentialequations}, and routines from the Trixi.jl package \cite{schlottkelakemper2021purely, ranocha2022adaptive, ranocha2021efficient}. Unless otherwise specified, we utilize the 4th order adaptive Runge-Kutta method implemented in \cite{rackauckas2017differentialequations}. The Julia codes used to generate the following numerical results are available at \url{https://github.com/raj-brown/quasi_1d_dgsem/}.

\subsection{Quasi-1D shallow water equations}

We begin by examining the accuracy of the proposed numerical methods for the quasi-1D shallow water equations with varying bathymetry and channel widths using the EC fluxes (\ref{eq:SWE_quasi_1D_flux}). We also check that the proposed schemes are well-balanced for spatially varying (including discontinuous) bathymetry and channel widths. 

\subsubsection{Convergence analysis}

We first examine convergence for the quasi-1D shallow water equations by comparing solutions \rnote{on uniformly refined grids} to a fine grid solution computed using $8000$ elements of degree $N=3$. We follow \cite{qian2018positivity} and utilize the following initial conditions:
\begin{align}
	\begin{aligned}
		a(x) &= \exp(\sin(2\pi x))\\
		h(x, 0) &= 3 + \exp(\cos(2\pi x))\\
		b(x) &= \sin^2(\pi x)\\
		u(x, 0) &= \frac{1}{ah}\sin(\cos(2\pi x))
	\end{aligned}
\end{align}

Table~\ref{table:swe_convergence} shows the computed $L^2$ errors for degree $N$ meshes of $\mathrm{K}$ elements. We observe that the rate of convergence appears to approach $O(h^{N+1})$ as the mesh is \rnote{uniformly} refined. Optimal rates of convergence are also observed when testing with manufactured solution in Table~\ref{table:swe_convergence_manufactured} using the following \bnote{initial condition}:
\begin{align}
    \bnote{h(x, t)} &\bnote{= 3 + \frac{1}{10}e^{\cos(2\pi x)}e^{-t}}\\
    \bnote{u(x, t)} &\bnote{= \sin(\cos(2 \pi x)) e^{-t}}
\end{align} 
\bnote{Source terms are computed using forward mode automatic differentiation via ForwardDiff.jl \cite{revels2016forward}.}

\begin{table}[!h]
\centering
	\begin{tabular}{|c|c|c|c|c|c|c|c|c|}
		\hline & \multicolumn{2}{|c|}{$N=1$} & \multicolumn{2}{c|}{$N=2$} & \multicolumn{2}{c|}{$N=3$} & \multicolumn{2}{c|}{$N=4$} \\
		\hline $\mathrm{K}$ & $L^2$ error & Rate & $L^2$ error & Rate & $L^2$ error & Rate & $L^2$ error & Rate  \\
		\hline 2 & $1.43 $ &-& $1.22$ &-& $7.05  \times 10^{-1} $ &-& $ 4.07  \times 10^{-1}$ &-  \\
		4 & $1.26 $ & 0.19&  $3.0 \times 10^{-1}$ & 2.05 &  $ 1.18  \times 10^{-1}$ & 2.61 & $3.28 \times 10^{-2}$  & 3.63 \\
		8 & $5.14\times 10^{-1}$ & 1.29 & $1.00 \times 10^{-1} $  & 1.56 & $1.48 \times 10^{-2}$  & 2.97& $1.89\times 10^{-3}$ & 4.12 \\
		16 & $2.01 \times 10^{-1} $ & 1.35 & $1.58 \times 10^{-2} $ & 2.67  & $6.90 \times 10^{-4}$ & 4.41 & $1.82\times 10^{-4}$ & 3.37 \\
		32 & $7.21  \times 10^{-2} $ & 1.48 & $2.53 \times 10^{-3}  $ & 2.64  & $7.88 \times 10^{-5}$ &  3.12& $6.98  \times 10^{-6}$  & 4.71\\
		\hline
	\end{tabular}
	\caption{ $L^2$ error and convergence rates for the quasi-1D shallow water equations with a fine-grid reference solution.}
	\label{table:swe_convergence}
\end{table}

\begin{table}
\centering
	\begin{tabular}{|c|c|c|c|c|c|c|}
		\hline & \multicolumn{2}{|c|}{$N=1$} & \multicolumn{2}{c|}{$N=2$} & \multicolumn{2}{c|}{$N=3$} \\
		\hline $\mathrm{K}$ & $L^2$ error & Rate & $L^2$ error & Rate & $L^2$ error & Rate \\
		\hline 16 & $4.80  \times 10^{-1}$ &-& $8.46 \times 10^{2}$ &-& $8.53  \times 10^{-3} $ &- \\
		32 & $1.30 \times 10^{-1} $ & 1.88&  $9.89  \times 10^{-3}$ & 3.10  &  $ 4.73  \times 10^{-4}$ & 4.17  \\
		64& $3.44 \times 10^{-2}$ & 1.92 & $1.23 \times 10^{-3} $  & 3.00  & $2.96 \times 10^{-5}$  & 4.00\\
		128 & $8.95 \times 10^{-3} $ & 1.94 & $1.54 \times 10^{-4} $ & 2.99 & $1.87 \times 10^{-6}$ & 3.99 \\
		256 & $2.28  \times 10^{-3} $ & 1.97& $1.93  \times 10^{-5}  $ & 3.00  & $1.17  \times 10^{-7}$ &  4.00 \\
		\hline
	\end{tabular}
	\begin{tabular}{|c|c|c|c|c|}
		\hline & \multicolumn{2}{c|}{$N=4$}  & \multicolumn{2}{c|}{$N=5$} \\
		\hline $\mathrm{K}$ & $L^2$ error & Rate & $L^2$ error & Rate \\
		\hline 16 & $ 1.00   \times 10^{-3}$ &- & $1.39  \times 10^{-4}$ & -  \\
		32 & $3.45 \times 10^{-5}$  & 4.86 &  $1.82 \times 10^{-6}$ & 6.25   \\
		64 & $1.09\times 10^{-6}$ & 4.98 & $2.84  \times 10^{-8}$ &  6.00   \\
		128 & $3.39  \times 10^{-8}$ & 5.01 & $4.48 \times 10^{-10}$  & 5.99 \\
		256 & $1.05 \times 10^{-9}$  & 5.01 & $7.04 \times 10^{-11}$ & 5.99 \\
		\hline
	\end{tabular}	
	\caption{ $L^2$ error and convergence rates for the quasi-1D shallow water equations with a manufactured solution.}
	\label{table:swe_convergence_manufactured}
\end{table}


\subsubsection{Verification of lake-at-rest well-balancedness}


We now consider a test of well-balancedness. For this experiment, we utilize the well-balanced Lax-Friedrichs penalization given by (\ref{eq:wb_lxf}). We consider both continuous and discontinuous channel widths and bottom topography in the domain $[0,1]$. The continuous bottom topography is given by
\begin{align}
b(x)= \begin{cases}\frac{1}{4}\LRp{1+\cos \LRp{10 \pi\LRp{x-\frac{1}{2}}}}, & \text { if } 0.4 \leq x \leq 0.6 \\ 0, & \text { otherwise }\end{cases}
\label{wb1}
\end{align}
The channel with continuous varying width $a(x)$ takes the form of
\begin{align}
a(x)= \begin{cases}1-\sigma_0\left(1+\cos \left(2 \pi \frac{x-\left(x_l+x_r\right) / 2}{x_r-x_l}\right)\right), & \text { if } x \in\left[x_l, x_r\right] \\ 1, & \text { otherwise }\end{cases}
\label{wb2}
\end{align}
where $x_l$ and $x_r$ are the left and right boundary of the contraction, and $1-2 a_0$ represents the minimum width of the channel at the point $\left(x_l+x_r\right) / 2$. In this example, we choose 
\[
x_l=0.25, \qquad x_r=0.75, \qquad a_0=0.2.
\]
The discontinuous channel width and bottom topography are given by
\[
b(x) = \begin{cases}
\frac{1}{2}  & x > \frac{1}{2}\\
0 & \text{ otherwise }
\end{cases}, \qquad
a(x)= \begin{cases}
1-\sigma_0\left(1+\cos \left(2 \pi \frac{x-\left(x_l-x_r\right) / 2}{x_r-x_l}\right)\right) & \text { if } x \in\left[x_l, x_r\right] \\ 
\frac{1}{2} & \text{ if } x > x_r\\
1, & \text { otherwise }\end{cases}
\]
For both continuous and discontinuous cases, the initial condition is the steady state lake-at-rest solution
$$
h+b=1, \qquad Q=a h u=0,
$$
and periodic boundary conditions are used. We discretize the domain using 200 uniform cells of degree $N=3$ and evolve the solution to final time $T_{\rm final}=1$ using the 4-stage 4th order Runge-Kutta method with sufficiently small time-step. Computed $L^1, L^2$, and $L^\infty$ errors are shown in Table~\ref{table:well_balanced}, and we observe that each computed error is close to machine precision for both continuous and discontinuous channel widths $a$ and bottom bathymetry $b$. 

\begin{table}
\centering
	\begin{tabular}{||c | c | c |c||} 
		\hline
		 Case & $L^1$ error & $L^2-$ error & $L^{\infty}$ error \\ [0.5ex] 
		\hline
		Continuous $b$ and $a$ & $9.19 \times 10^{-15}$ & $5.92 \times 10^{-11}$  & $2.01\times 10^{-13}$  \\
		\hline
		Discontinuous $b$ and $a$ & $1.46 \times 10^{-14}$ & $3.15 \times 10^{-14}$  & $ 1.65\times 10^{-16}$  \\
	  \hline 
	\end{tabular}
\caption{Computed errors for the well-balanced test}
\label{table:well_balanced}
\end{table}


\subsubsection{Converging-diverging channel}

We conclude with an experiment on steady transcritical flow in a converging-diverging channel \cite{garcia19921, vazquez1999improved, qian2018positivity}. The domain is $[0, 500]$, the bottom bathymetry is flat with $b=0$, and the converging-diverging channel geometry is given by (\ref{eq:swe_conv_div_channel})\footnote{The formula for the channel geometry follows \cite{vazquez1999improved}; the expression in \cite{qian2018positivity} appears to be slightly different.}
\begin{equation}
a(x) = \begin{cases}
5 - 0.7065 \LRp{1 + \cos\LRp{2\pi \frac{x-250}{300}}}, & x \in [100, 450]\\
5, & \text{otherwise}
\end{cases}
\label{eq:swe_conv_div_channel}
\end{equation}
The initial conditions are $h = 2, ahu = 20$, and boundary conditions are taken to be $ahu = 20$ at the inflow $x=0$ and $h = 1.85$ at the outflow $x = 500$. \rnote{The solutions are discretized using 200 and 400 uniform elements} of degree $N=2$. The water height $h$ and computed Froude number $|u| / c$ where $c = \sqrt{g h}$ are shown in Figure~\ref{fig:swe_conv_div}, along with reference solution values taken from \cite{vazquez1999improved}. \rnote{The solutions show some discrepancies at the shock, which may be due to the smoothing out of the shock profile by the slope limiter along the upstream direction in \cite{vazquez1999improved}.}

\begin{figure}[h]
\centering
\subfloat[Water height $h$ (200 elements)]{\includegraphics[width=.49\textwidth]{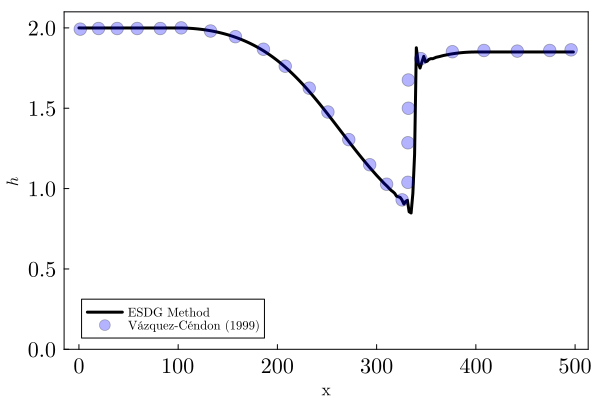}}
\subfloat[Froude number (200 elements)]{\includegraphics[width=.49\textwidth]{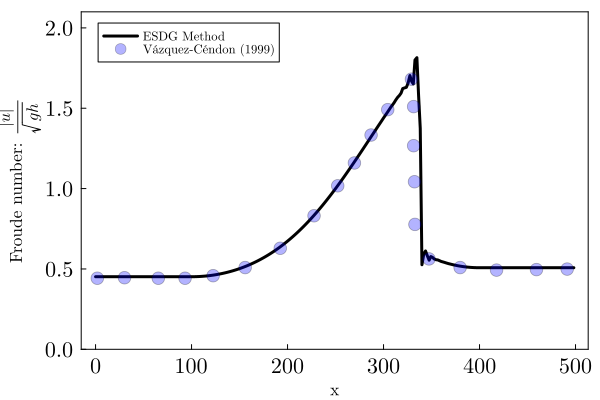}}\\
\subfloat[Water height $h$ (400 elements)]{\includegraphics[width=.49\textwidth]{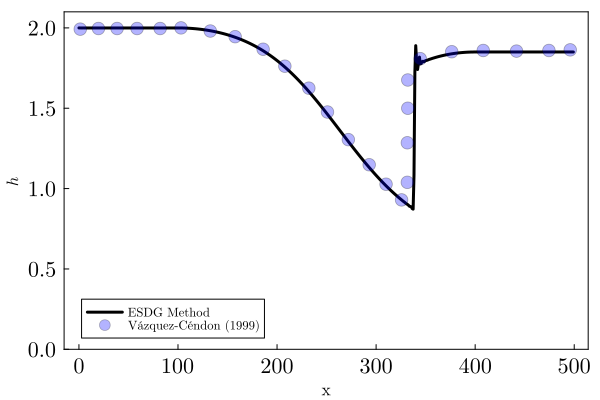}}
\subfloat[Froude number (400 elements)]{\includegraphics[width=.49\textwidth]{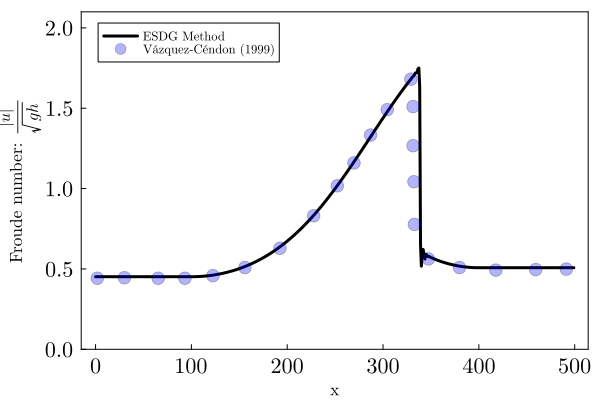}}
\caption{Water height and Froude number for the converging-diverging channel problem.}
\label{fig:swe_conv_div}
\end{figure}

\subsection{Quasi-1D compressible Euler equations}

In this section, we examine the behavior of high order entropy stable DGSEM schemes using the EC fluxes (\ref{eq:Euler_quasi_1D_flux}) for the compressible Euler equations. For all problems, $\gamma=1.4$. 

\subsubsection{Entropy conservation verification}

We first verify the entropy conservation of the proposed scheme by evolving a discontinuous initial condition up to final time $T_{\rm final} = 2$. The domain is $[-1,1]$ with periodic boundary conditions, and the interface flux is taken to be the entropy conservative flux (\ref{eq:Euler_quasi_1D_flux}). Together, this yields an entropy conservative high order scheme. The initial state is given by left and right data
\[
\LRp{\rho_L, u_L, p_L} = (3.4718, -2.5923, 5.7118), \qquad \LRp{\rho_R, u_R, p_R} = (2, -3, 2.639),
\]
which lie on the left/right of a discontinuity at $x = 0$ \bnote{(as well as a discontinuity across the boundaries since the domain is periodic}). The nozzle width is also taken to be a discontinuous function 
\[
a(x) = \begin{cases}
1 	& x < 0\\
1.5 	& x \geq 0
\end{cases}.
\]
Figure~\ref{fig:ec} shows the evolution of the entropy residual $\sum_k \fnt{v}^T\fnt{M}_k\td{\fnt{u}}{t}$ for a degree $N=3$ simulation on a mesh of $64$ elements. As expected, the residual remains between $O(10^{-12})$ and $O(10^{-14})$ for the duration of the simulation. 

\begin{figure}[h]
\centering
\includegraphics[width=.49\textwidth]{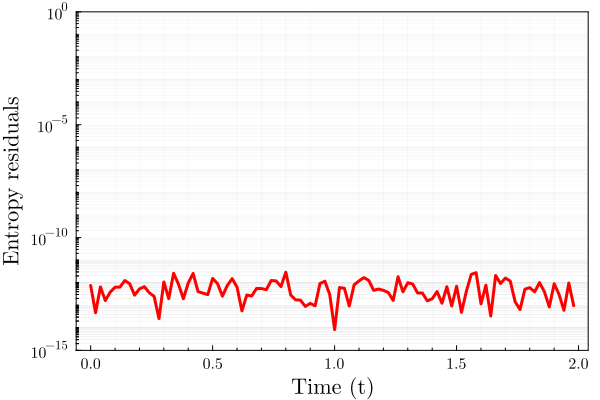}
\caption{Entropy residual for an entropy conservative simulation of the quasi-1D compressible Euler equations.}
\label{fig:ec}
\end{figure}

\subsubsection{Convergence analysis}

Next, we perform convergence tests for the compressible Euler equation using both a reference solution on a highly refined grid and the method of manufactured solutions. We first examine convergence against a fine grid solution on a degree \bnote{$N=2$ mesh of $24000$} elements. The initial condition for this method is given by a small smooth perturbation of a constant state
\begin{align}
\label{ic_euler_ref}
\begin{aligned}
\rho(x, 0) &= 1 - \frac{1}{10}\LRp{1 + \sin\LRp{2 \pi\LRp{x - \frac{1}{10}}}}, \\
u(x, 0) &= 0, \qquad p(x, 0) = \rho^{\gamma}.
\end{aligned} 
\end{align}
The nozzle width $a(x)$ is given by
\[
a(x) = 1 - \frac{1}{5} \LRp{1 + \cos\LRp{2 \pi \LRp{x - \frac{1}{2}} }}
\]
\bnote{We impose values of the analytical solution at domain boundaries.} The solution is evolved until final time $T_{\rm final} =1 / 10$. Table~\ref{tab:ref_sol} shows computed convergence rates, which approach the optimal rate of $h^{N+1}$ \bnote{for $N > 1$. For $N=1$, we observe a sub-optimal rate of convergence of 1.5. } 

\begin{table}
\centering
\bnote{
	\begin{tabular}{|c|c|c|c|c|c|c|}
		\hline & \multicolumn{2}{|c|}{$N=1$} & \multicolumn{2}{c|}{$N=2$} & \multicolumn{2}{c|}{$N=3$} \\
		\hline $\mathrm{K}$ & $L^2$ error & Rate & $L^2$ error & Rate & $L^2$ error & Rate  \\
		\hline 
		2 	& $2.054 \times 10^{-2}$ 	& - 	 	& $1.003 \times 10^{-2}$ & - 			& $1.267 \times 10^{-2}$ 		& - \\
		4 	& $1.335 \times 10^{-2}$	& 0.62 	& $1.055 \times 10^{-2}$ & -0.074		& $4.106 \times 10^{-3}$ 		& 1.63 \\
		8 	& $8.004 \times 10^{-3}$	& 0.74 	& $3.672 \times 10^{-3}$ & 1.52 		& $5.704 \times 10^{-4}$ 		& 2.85 \\
		16 	& $3.798 \times 10^{-3}$	& 1.08 	& $6.631 \times 10^{-4}$ & 2.47 		& $3.884 \times 10^{-5}$ 		& 3.88 \\
		32 	& $1.627 \times 10^{-3}$	& 1.22 	& $8.613 \times 10^{-5}$ & 2.94 		& $2.454 \times 10^{-6}$	 	& 3.98 \\
		\hline
	\end{tabular}
	\begin{tabular}{|c|c|c|c|c|c|c|c|c|}
		\hline & \multicolumn{2}{c|}{$N=4$} & \multicolumn{2}{c|}{$N=5$} \\
		\hline $\mathrm{K}$ & $L^2$ error & Rate & $L^2$ error & Rate \\
		\hline 
		2 	& $7.481 \times 10^{-3}$ 	& - 		& $5.173 \times 10^{-3}$ 		& - \\
		4 	& $1.188 \times 10^{-3}$ 	& 2.66 	& $3.524 \times 10^{-4} $		& 3.88 \\
		8 	& $7.096 \times 10^{-5}$ 	& 4.06 	& $9.723  \times 10^{-6}	$	& 5.18 \\
		16 	& $2.458 \times 10^{-6}$ 	& 4.85 	& $1.594  \times 10^{-7}	$	& 5.93 \\
		32 	& $8.218 \times 10^{-8}$ 	& 4.90 	& $2.525  \times 10^{-9}	$	& 5.98 \\
		\hline
	\end{tabular}	
\caption{$L^2$ errors and convergence rates for the quasi-1D compressible Euler equations relative to a fine-grid reference solution.}
\label{tab:ref_sol}
}
\end{table}

\bnote{We also verify high order rates of convergence for a non-uniform grid in Table~\ref{tab:ref_sol_nonuniform}. The vertex locations of the non-uniform grid are constructed through a mapping of vertex coordinates $x_i$ of a uniform grid on $[-1, 1]$ via $x_i - 0.3 \sin(\pi x_i)$. We observe again that the errors converge at rates approaching optimal rates of convergence for $N > 1$. For $N=1$, the computed rate of convergence again appears to be approaching $1.5$, though at a slower rate than for uniform meshes.}
\begin{table}
\centering
\bnote{
	\begin{tabular}{|c|c|c|c|c|c|c|}
		\hline & \multicolumn{2}{|c|}{$N=1$} & \multicolumn{2}{c|}{$N=2$} & \multicolumn{2}{c|}{$N=3$} \\
		\hline $\mathrm{K}$ & $L^2$ error & Rate & $L^2$ error & Rate & $L^2$ error & Rate  \\
		\hline 
		4 	& $2.054 \times 10^{-2}$ 	& - 	 	& $1.003 \times 10^{-2}$ & - 			& $1.267 \times 10^{-2}$ 		& - \\
		8 	& $1.335 \times 10^{-2}$	& 0.62 	& $1.055 \times 10^{-2}$ & -0.074		& $4.106 \times 10^{-3}$ 		& 1.63 \\
		16 	& $8.004 \times 10^{-3}$	& 0.74 	& $3.672 \times 10^{-3}$ & 1.52 		& $5.704 \times 10^{-4}$ 		& 2.85 \\
		32 	& $3.798 \times 10^{-3}$	& 1.08 	& $6.631 \times 10^{-4}$ & 2.47 		& $3.884 \times 10^{-5}$ 		& 3.88 \\
		64 	& $1.627 \times 10^{-3}$	& 1.22 	& $8.613 \times 10^{-5}$ & 2.94 		& $2.454 \times 10^{-6}$	 	& 3.98 \\
		\hline
	\end{tabular}
	\begin{tabular}{|c|c|c|c|c|c|c|c|c|}
		\hline & \multicolumn{2}{c|}{$N=4$} & \multicolumn{2}{c|}{$N=5$} \\
		\hline $\mathrm{K}$ & $L^2$ error & Rate & $L^2$ error & Rate \\
		\hline 
		4 	& $7.481 \times 10^{-3}$ 	& - 		& $5.173 \times 10^{-3}$ 		& - \\
		8 	& $1.188 \times 10^{-3}$ 	& 2.66 	& $3.524 \times 10^{-4} $		& 3.88 \\
		16 	& $7.096 \times 10^{-5}$ 	& 4.06 	& $9.723  \times 10^{-6}	$	& 5.18 \\
		32 	& $2.458 \times 10^{-6}$ 	& 4.85 	& $1.594  \times 10^{-7}	$	& 5.93 \\
		64 	& $8.218 \times 10^{-8}$ 	& 4.90 	& $2.525  \times 10^{-9}	$	& 5.98 \\
		\hline
	\end{tabular}	
\caption{$L^2$ errors and convergence rates for the quasi-1D compressible Euler equations on a non-uniformly spaced grid using a fine-grid reference solution.}
\label{tab:ref_sol_nonuniform}
}
\end{table}

We next examine convergence for the following manufactured solution: 
\begin{align}\label{sol_euler_manuf}
\begin{aligned}
\rho(x, t) &= \LRp{1 + \frac{1}{10}\sin(2 \pi x) + \frac{1}{10}\cos(2 \pi x)} \exp(-t),\\
u(x, t)     &= \LRp{1 + \frac{1}{10}\sin(2 \pi x) + \frac{1}{10}\cos(2 \pi x)} \exp(-t),\\
p(x, t)     &= \LRp{1 + \frac{1}{10}\sin(2 \pi x) + \frac{1}{10}\cos(2 \pi x)} \exp(-t),\\    
\end{aligned} 
\end{align}
The nozzle width $a(x)$ is given by
\[
a(x) = 1 - \frac{1}{10} \LRp{1 + \cos\LRp{2 \pi \LRp{x - \frac{1}{2}} }}
\]
Source terms are computed using ForwardDiff.jl \cite{revels2016forward}, and fine-grid solutions are interpolated using Interpolations.jl \cite{tim_holy_2023_8066592}. Table~\ref{tab:manufacture-sol} reports both $L^2$ errors and convergence rates for the manufactured solution at final time $T_{\rm final} =1/10$. For odd orders, the computed rates of convergence approach $O(h^{N+1})$. However, for even orders, we appear to observe a suboptimal rate of convergence. This will be analyzed and investigated in future work. 

\begin{table}
\centering
\begin{tabular}{|c|c|c|c|c|c|c|c|}
\hline & \multicolumn{2}{|c|}{$N=1$} & \multicolumn{2}{c|}{$N=2$} & \multicolumn{2}{c|}{$N=3$} \\ 
\hline $N_{\rm elem}$ & $L^2$ error & Rate & $L^2$ error & Rate & $L^2$ error & Rate \\
\hline 5 & $8.514 \times 10^{-1}$      &-& 		  $8.056 \times 10^{-1}$   &- 		 & $3.518 \times 10^{-2}$   &-\\
	 10 & $ 3.107 \times 10^{-1}$   & $1.45$  & $2.354 \times 10^{-1}$   & $2.844$   & $2.307 \times 10^{-3}$   & $3.930$ \\
  	 20 & $8.833 \times 10^{-2}$    & $1.81$  & $3.277 \times 10^{-2}$   & $2.731$   & $3.843 \times 10^{-4}$   & $3.942$ \\
	 40 &  $2.280 \times 10^{-2}$ & $1.95$    & $4.936 \times 10^{-3}$   & $2.537$   & $2.619 \times 10^{-5}$   & $3.715$ \\
 	 80 & $5.712 \times 10^{-3}$  & $1.97$    & $8.505 \times 10^{-4}$   & $2.392$   & $7.910 \times 10^{-7}$   & $3.853$ \\
\hline
\end{tabular}
\newline
\vspace{.5em}
\newline
\begin{tabular}{|c|c|c|c|c|}
\hline & \multicolumn{2}{c|}{$N=4$} & \multicolumn{2}{c|}{$N=5$} \\
\hline $N_{\rm elem}$ & $L^2$ error & Rate & $L^2$ error & Rate \\
\hline 5 	& $5.539 \times 10^{-3}$  	&- 		& $3.961 \times 10^{-2}$ 	 &-  \\ 
	10 	& $2.100 \times 10^{-4}$  	& $4.72$  & $2.749 \times 10^{-3}$ 	& $5.64$ \\
	20 	& $1.029 \times 10^{-5}$ 	& $4.35$  & $5.790 \times 10^{-5}$ 	& $5.85$   \\ 
	40 	& $5.083 \times 10^{-7}$ 	& $4.34$  & $1.169 \times 10^{-6}$  	& $5.95$ \\ 
	80 	& $2.599 \times 10^{-8}$ 	& $4.29$ 	& $1.981 \times 10^{-8}$ 	& $6.00$ \\
\hline
\end{tabular}
\caption{$L^2$ errors and convergence rates for 1D quasi-Euler equation \bnote{using a manufactured solution.}}
\label{tab:manufacture-sol}
\end{table}

\subsubsection{Convergent-divergent nozzle flow}

Finally, we consider both a subsonic and transonic flow through a Laval (e.g., a convergent-divergent) nozzle. Following \cite{osher1983upwind, poinsot1992boundary}, we impose subsonic inflow boundary conditions on density and pressure and subsonic outflow conditions on pressure only. Analytical expressions for steady solutions are given in \cite{pulliam2014fundamental}, and we take the initial condition to be a constant which satisfies exact solution values at the inflow. The final time is taken to be $T_{\rm final} = 5$, such that the solution has reached a steady state. For this problem, we utilize the 9-stage, fourth order low-storage adaptive time-stepping scheme in \cite{ranocha2023error}. 

\begin{figure}
\centering
\includegraphics[width=.49\textwidth]{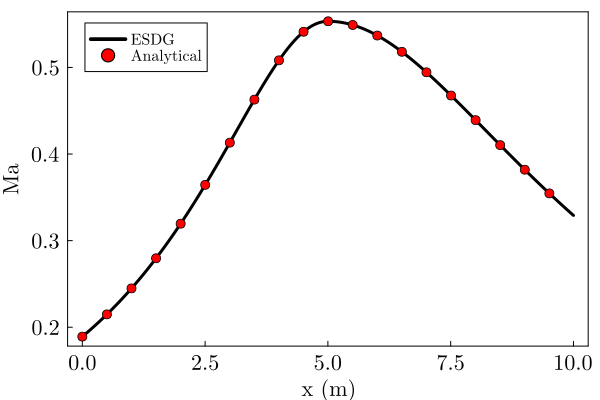}
\includegraphics[width=.49\textwidth]{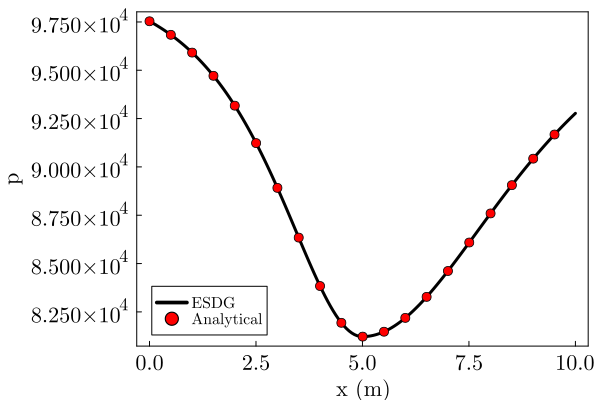}
\caption{Mach number and pressure for subsonic flow through a nozzle.}
\label{fig:sub_sonic_cf}
\end{figure}

Figure~\ref{fig:sub_sonic_cf} and \ref{fig:trans_sonic_cf} show results for subsonic and transonic settings, respectively. For subsonic flow in Figure~\ref{fig:sub_sonic_cf}, we use degree a $N=3$ approximation on a mesh of $16$ uniform elements. The solution does not contain shocks and is accurately approximated. For transonic flow in Figure~\ref{fig:trans_sonic_cf}, we use a degree $N=3$ approximation on a mesh of $64$ uniform elements. Since the solution contains a $C^0$ kink and a shock discontinuity, the solution contains oscillations, which can be removed by postprocessing, shock capturing/limiting, or by adding physical viscosity.

\begin{figure}
\centering
\includegraphics[width=.49\textwidth]{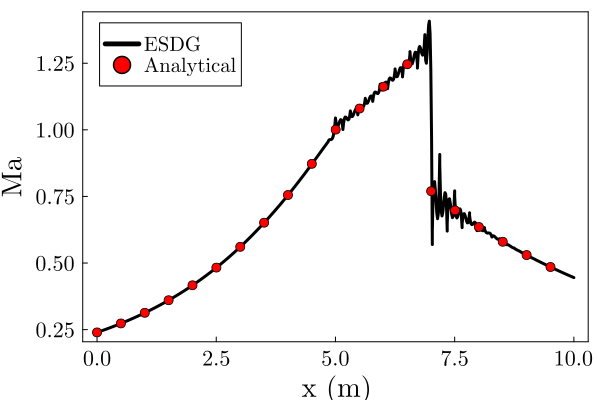}
\includegraphics[width=.49\textwidth]{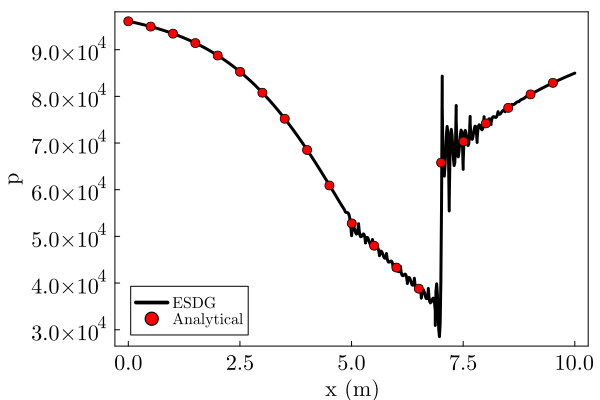}
\caption{Mach number and pressure for transonic flow through a nozzle.}
\label{fig:trans_sonic_cf}
\end{figure}

\section{Conclusion} 
\label{sec:con}

In this work, we introduce entropy stable high order schemes based on flux differencing for the quasi-1D shallow water and quasi-1D compressible Euler equations. Because these equations are not in conservation form, the entropy conservative fluxes used to construct such schemes are no longer symmetric. We introduce a new condition for entropy conservation/stability which accommodates non-symmetric fluxes and simplifies the entropy analysis. We also theoretically justify the high order accuracy, conservation, and well-balanced properties of the new schemes since existing proofs of high order accuracy and conservation rely on symmetry of the fluxes. Future work will focus on incorporating positivity preservation into the proposed schemes.

\section*{Acknowledgements}
Jesse Chan gratefully acknowledges support from National Science Foundation under awards DMS-1943186 and DMS-223148. Khemraj Shukla gratefully acknowledges support from the Air Force Office of Science and Research (AFOSR) under OSD/AFOSR MURI Grant FA9550-20-1-0358. The authors also thank Dr.\ David Craig Penner and Prof.\ David Zingg for helpful discussions, as well as for sharing their Laval nozzle reference solution codes.

\appendix 
\section{Continuous entropy analysis for sufficiently regular solutions}
\label{appendix:entropy_analysis}

\subsection{Quasi-1D shallow water}

Under the assumption that $\frac{\partial a}{\partial t} = 0$, we can rewrite Eq.\ (\ref{eq:SWE_quasi_1D}) as
\begin{align}
a\frac{\partial }{\partial t}\begin{bmatrix}
h\\
hu
\end{bmatrix}+
a\frac{\partial }{\partial x}\begin{bmatrix}
hu\\
hu^2+gh^2/2
\end{bmatrix} + 
\frac{\partial a}{\partial x}\begin{bmatrix}
hu\\
hu^2 + gh^2/2
\end{bmatrix} - 
\frac{\partial a}{\partial x}\begin{bmatrix}
0\\
gh^2/2
\end{bmatrix}= 
\begin{bmatrix}
0\\
- gahb_x
\end{bmatrix}.
\end{align}
Define the following group variables  
\begin{align}
\bm{u} = \begin{bmatrix}
h\\
hu
\end{bmatrix}, \quad \bm{f}(\bm{u}) = \begin{bmatrix}
hu\\
hu^2+gh^2/2
\end{bmatrix}, \quad \bm{P} = \begin{bmatrix}
0\\
gh^2/2
\end{bmatrix}, \quad \bm{S} = \begin{bmatrix}
0\\
-gh^2b_x/2
\end{bmatrix}
.
\label{eq:swe_q1D_vec_def}
\end{align}
Then, we have
\begin{align}
a\frac{\partial \bm{u}}{\partial t}+
a\frac{\partial }{\partial x}\bm{f}(\bm{u})+
\frac{\partial a}{\partial x}(\bm{f}(\bm{u}) - \bm{P})= a\bm{S}.
\label{eq:SWE_q1D_vec}
\end{align}
Let the entropy, $S$, and entropy flux, $F$ be defined for the standard 1D shallow water equations with bathymetry $b$ from \cite{wintermeyer2018entropy}
\begin{align}
S(\bm{u}) = \frac{1}{2}hu^2+\frac{1}{2}gh^2 + ghb, \qquad F(\bm{u}) = \frac{1}{2}hu^3+ghu(h+b).
\end{align}
The \rnote{components of the} entropy variables $\bm{v}$ are then
\begin{align}
    v_1 = gh-\frac{1}{2}u^2 + gb, \qquad v_2 = u.
\end{align}
Multiplying by $\bm{v}^T$ in Eq.\ (\ref{eq:SWE_q1D_vec}), we have
\begin{align}
a\LRp{\frac{\partial S(\bm{u})}{\partial t}+ \frac{\partial F(\bm{u})}{\partial x}}+
\bm{v}^T\frac{\partial a}{\partial x}(\bm{f}(\bm{u}) - \bm{P})= 0.
\label{eq:SWE_q1D_ent}
\end{align}
We also have
\begin{align}
\nonumber
\bm{v}^T(\bm{f}(\bm{u}) - \bm{P}) =& (gh-\frac{1}{2}u^2 + gb)(hu) + hu^3 = F(\bm{u}).
\end{align}
Then, Eq.\ (\ref{eq:SWE_q1D_ent}) can be written as
\begin{align}
a\LRp{\frac{\partial S(\bm{u})}{\partial t}+ \frac{\partial F(\bm{u})}{\partial x}}+
\frac{\partial a}{\partial x}F(\bm{u})= 0\\
\Longrightarrow \frac{\partial aS(\bm{u})}{\partial t}+ a\frac{\partial F(\bm{u})}{\partial x}+ \frac{\partial a}{\partial x}F(\bm{u})= 0\\
\Longrightarrow \frac{\partial aS(\bm{u})}{\partial t}+ \frac{\partial aF(\bm{u})}{\partial x} = 0.
\label{eq:SWE_q1D_ent_final}
\end{align}

\subsection{Quasi-1D compressible Euler}

We now perform a similar analysis for the quasi-1D compressible Euler equations. Assuming that the width of the channel, $a(x)$, does not change over the time, we can write Eq.\ (\ref{eq:Euler_q1D}) as
\begin{align}
a\frac{\partial }{\partial t}\begin{bmatrix}
\rho\\
\rho u\\
E
\end{bmatrix}+
a\frac{\partial }{\partial x}\begin{bmatrix}
\rho u\\
\rho u^2 + p\\
u(E+p)
\end{bmatrix}+
\frac{\partial a}{\partial x}\begin{bmatrix}
\rho u\\
\rho u^2 + p\\
u(E+p)
\end{bmatrix}-
\frac{\partial a}{\partial x}\begin{bmatrix}
0\\
p\\
0
\end{bmatrix}= 0.
\label{eq:Euler_q1D2}
\end{align}
Let
\begin{align}
\bm{u} =\begin{bmatrix}
\rho\\
\rho u\\
E
\end{bmatrix}, \qquad \bm{f}(\bm{u}) =\begin{bmatrix}
\rho u\\
\rho u^2 + p\\
u(E+p)
\end{bmatrix}, \qquad \bm{P} = 
\begin{bmatrix}
0\\
p\\
0
\end{bmatrix}.
\label{eq:Euler_q1D_vec_def}
\end{align}
We have
\begin{align}
a\frac{\partial \bm{u}}{\partial t}+
a\frac{\partial }{\partial x}\bm{f}(\bm{u})+
\frac{\partial a}{\partial x}(\bm{f}(\bm{u}) - \bm{P})= 0.
\label{eq:Euler_q1D_vec}
\end{align}
Let the entropy, $S$, and entropy flux, $F$ be defined for the standard 1D compressible Euler equations, such that
\begin{align}
S(\bm{u}) = \frac{-\rho s}{\gamma -1}, \qquad s = \log\LRp{\frac{p}{\rho^\gamma}}, \qquad F(\bm{u}) = \frac{-\rho u s}{\gamma -1}.
\end{align}
The \rnote{components of the} entropy variables $\bm{v}$ are then
\begin{align}
v_1 = \frac{\gamma - s}{\gamma - 1} - \frac{\rho u^2}{2p}, \qquad v_2= \frac{\rho u}{p}, \qquad v_{3} = -\frac{\rho}{p}.
\end{align}
Multiplying by $\bm{v}^T$ in Eq.\ (\ref{eq:Euler_q1D_vec}), we have
\begin{align}
a\LRp{\frac{\partial S(\bm{u})}{\partial t}+ \frac{\partial F(\bm{u})}{\partial x}}+
\bm{v}^T\frac{\partial a}{\partial x}(\bm{f}(\bm{u}) - \bm{P})= 0.
\label{eq:Euler_q1D_ent}
\end{align}
We also have
\begin{align}
\nonumber
\bm{v}^T(\bm{f}(\bm{u}) - \bm{P}) =& \LRp{\frac{\gamma - s}{\gamma - 1} - \frac{\rho u^2}{2p}}(\rho u) + 
\frac{\rho u}{p}(\rho u^2) -\frac{\rho}{p}\LRp{\frac{1}{2}\rho u^2 + \frac{p}{\gamma-1} + p}u\\
=&\frac{\gamma \rho u}{\gamma - 1} - \frac{\rho u s}{\gamma - 1} - \frac{\rho^2 u^3}{2p} + \frac{\rho^2 u^3}{p} - \frac{\rho^2 u^3}{2p} - \frac{\rho u}{\gamma - 1} - \rho u\\
=& \frac{\gamma \rho u}{\gamma - 1} - \frac{\rho u s}{\gamma - 1} - \frac{\rho u}{\gamma - 1} - \frac{\rho u(\gamma - 1)}{\gamma - 1} = -\frac{\rho u s}{\gamma - 1} = F.
\end{align}
Then, Eq.\ (\ref{eq:Euler_q1D_ent}) can be written as
\begin{align}
a\LRp{\frac{\partial S(\bm{u})}{\partial t}+ \frac{\partial F(\bm{u})}{\partial x}}+
\frac{\partial a}{\partial x}F(\bm{u})= 0\\
\Longrightarrow \frac{\partial aS(\bm{u})}{\partial t}+ a\frac{\partial F(\bm{u})}{\partial x}+
\frac{\partial a}{\partial x}F(\bm{u})= 0.\\
\Longrightarrow \frac{\partial aS(\bm{u})}{\partial t}+ \frac{\partial aF(\bm{u})}{\partial x} = 0.
\label{eq:Euler_q1D_ent_final}
\end{align}

\bibliographystyle{elsarticle-num}
\bibliography{reference.bib}

\end{document}